\documentclass[a4paper, 11pt]{amsart}
\usepackage{geometry}
\usepackage{indentfirst}  
\usepackage{times}
\usepackage{mathrsfs}
\usepackage{latexsym}
\usepackage[dvips]{graphics}
\usepackage{enumitem}
\usepackage{epsfig}
\usepackage{hyperref, amsmath, amsthm, amsfonts, amscd, flafter,epsf}
\usepackage{amsmath,amsfonts,amsthm,amssymb,amscd}
\usepackage{color}
\usepackage[T1]{fontenc}
\usepackage{amssymb} 
\usepackage{stackrel}
\usepackage{mathtools} 
\usepackage{tikz} 
\usepackage[utf8]{inputenc}
\usepackage{cleveref}
\crefname{section}{§}{§§}
\Crefname{section}{§}{§§}
\usetikzlibrary{matrix,arrows,decorations.pathmorphing} 
\usetikzlibrary{chains} 

\tikzset{node distance=2em, ch/.style={circle,draw,on chain,inner sep=2pt},chj/.style={ch,join},every path/.style={shorten >=4pt,shorten <=4pt},line width=1pt,baseline=-1ex}

\newcommand\be{\begin{equation}}
\newcommand\ee{\end{equation}}
\newcommand\bea{\begin{eqnarray}}
\newcommand\eea{\end{eqnarray}}
\newcommand\bi{\begin{itemize}}
	\newcommand\ei{\end{itemize}}
\newcommand\ben{\begin{enumerate}}
	\newcommand\een{\end{enumerate}}

\newtheorem{theorem}{Theorem}[section]

\newtheorem{corollary}[theorem]{Corollary}
\newtheorem{lemma}[theorem]{Lemma}
\newtheorem{proposition}[theorem]{Proposition}

\theoremstyle{definition}
\newtheorem{definition}[theorem]{Definition}

\theoremstyle{remark}

\newtheorem{remark}[theorem]{Remark}
\newtheorem{assumption}[theorem]{Assumption}

\newcommand{\Z}{\ensuremath{\mathbb{Z}}}
\newcommand{\simrightarrow}{\stackrel{\sim}{\rightarrow}}
\newcommand{\Q}{\mathbb{Q}}

\newcommand{\N}{\mathbb{N}}

\newcommand{\cF}{\mathcal{F}}

\newcommand{\cO}{\mathcal{O}}
\newcommand{\fg}{\mathfrak{g}}
\newcommand{\ft}{\mathfrak{t}}
\newcommand{\fb}{\mathfrak{b}}

\newcommand{\fm}{\mathfrak{m}}

\newcommand{\cR}{\mathcal{R}}

\newcommand{\Gal}{\mathrm{Gal}}
\newcommand{\cG}{\mathcal{G}}
\newcommand{\cT}{\mathcal{T}}

\newcommand{\BdR}{\mathrm{B}_{\mathrm{dR}}}
\newcommand{\GL}{\mathrm{GL}}

\newcommand{\wt}{\mathrm{wt}}

\newcommand{\Tor}{\mathrm{Tor}}

\newcommand{\Hom}{\mathrm{Hom}}

\newcommand{\Spf}{\mathrm{Spf}}

\newcommand{\Sp}{\mathrm{Sp}}
\newcommand{\Fil}{\mathrm{Fil}}

\newcommand{\fX}{\mathfrak{X}}
\newcommand{\trivar}{X_{\mathrm{tri}}(\overline{r})}

\numberwithin{equation}{section}

\title{Companion points on the eigenvariety with non-regular weights}

\author{Zhixiang Wu}
\date{}
\begin{document}
\address{B\^atiment 307, Facult\'e d'Orsay, Universit\'e Paris-Saclay, 91405 Orsay Cedex, France}
\email{zhixiang.wu@universite-paris-saclay.fr}
\begin{abstract}
	We prove the existence of all companion points on the eigenvariety of definite unitary groups associated with generic crystalline Galois representations with possibly non-regular weights under the Taylor-Wiles hypothesis, based on the previous results of Breuil-Hellmann-Schraen in \cite{breuil2019local} in regular cases and the author in \cite{wu2021local} in non-regular cases.
\end{abstract}
\thanks{The author would like to thank his thesis advisor Benjamin Schraen for discussions.
}
\maketitle
%\tableofcontents
\section{Introduction}
Let $p$ be a prime number, then there exists pairs of ($p$-adic) elliptic modular eigenforms $(f,g)$ of level $\Gamma_0(p)\cap \Gamma_1(N)$ for some $p\nmid N$ such that $f$ and $g$ share the same eigenvalues for Hecke operators $T_{\ell}$ when $\ell\nmid pN$ (i.e. $f$ and $g$ are associated with the same $p$-adic Galois representation), but have \emph{different} non-zero eigenvalues for the $U_p$-operator. The results on the existence of such \emph{companion forms} for $p$-adic or mod-$p$ modular forms, as of Gross in \cite{gross1990tameness}, have many significant applications. For example, Buzzard and Taylor (\cite{buzzard1999companion}, and see \cite{buzzard2003analytic}) use Gross's results to prove the classicality of overconvergent $p$-adic weight one modular forms (hence certain cases of the Artin's conjecture), and their methods have been successfully generalized for Hilbert modular forms of parallel weight one, e.g., \cite{pilloni2017formes}, \cite{pilloni2016arithmetique} and \cite{sasaki2019integral}.\par
In \cite{hansen2017universal}, Hansen made a conjecture on the existence of all companion forms for finite slope overconvergent $p$-adic automorphic forms of general $\GL_n$ in the language of determining the set of \emph{companion points} on the eigenvariety that are associated with the same $p$-adic Galois representation but with possibly different $U_p$-eigenvalues or weights. Similar to the weight part of Serre's modularity conjecture, the recipes for companion forms are given by the $p$-adic local Galois representations. In fact, the conjecture on companion points is closely related to Breuil's locally analytic socle conjecture in \cite{breuil2016versI}\cite{breuil2015versII} from the point of view of the local-global compatibility in the locally analytic aspect of the $p$-adic local Langlands program.\par
We will work in the setting of definite unitary groups as Breuil. Let $F$ be a quadratic imaginary extension of a totally real field $F^+$. Let $S_p$ be the set of places of $F^+$ that divide $p$. We assume that each $v\in S_p$ splits in $F$. Let $G$ be a definite unitary group of rank $n\geq 2$ over $F^+$ that is split over $F$ (so that $G(F^+\otimes_{\Q}\Q_p)\simeq \prod_{v\in S_p}\GL_n(F_v^{+})$). Then an eigenvariety $Y(U^p,\overline{\rho})$ of $G$, of certain tame level $U^p$ and localized at a modular absolutely irreducible $\overline{\rho}:\Gal(\overline{F}/F)\rightarrow \GL_n(\overline{\mathbb{F}}_p)$, is a rigid analytic space parametrizing pairs $(\rho,\underline{\delta})$ where $\rho:\Gal(\overline{F}/F)\rightarrow \GL_n(\overline{\mathbb{Q}}_p)$ are continuous representations which lift $\overline{\rho}$ and $\underline{\delta}=(\underline{\delta}_v)_{v\in S_p}=(\delta_{v,i})_{v\in S_p,i=1,\cdots,n}: \prod_{v\in S_p}((F^{+}_v)^{\times})^n\rightarrow \overline{\Q}_p^{\times}$ is a continuous character such that $\rho$ is associated with a finite slope overconvergent $p$-adic automorphic form of $G$ which has ``weight'' $\underline{\delta}\mid_{\prod_{v\in S_p}(\cO_{F_v^{+}}^{\times})^n}$ and has ``$U_p$-eigenvalues'' $\prod_{j=1}^i\delta_{v,j}(\varpi_{F_v^+})$ for $v\in S_p$, $i=1,\cdots n$ where $\varpi_{F_v^{+}}$ denotes some uniformizers. \par
Recall an algebraic character of $F_v^{+}$ has the form $(F_v^{+})^{\times}\rightarrow \overline{\Q}_p^{\times}:z\mapsto \prod_{\tau:F_v^{+}\hookrightarrow \overline{\Q}_p}\tau(z)^{k_{\tau}}$ for some $k_{\tau}\in \Z$. Now take a point $x=(\rho,\underline{\delta})\in Y(U^p,\overline{\rho})$ and assume that $\rho_v:=\rho\mid_{\Gal(\overline{F_{\widetilde{v}}}/F_{\widetilde{v}})}$ is crystalline for all $v\in S_p$ where $\widetilde{v}\mid v$ is a place of $F$ chosen for each $v\in S_p$. Then $\underline{\delta}$ is locally algebraic, i.e. $\underline{\delta}=\underline{\delta}_{\mathrm{alg}}\underline{\delta}_{\mathrm{sm}}$ where each $\delta_{\mathrm{alg},v,i}$ is algebraic and $\delta_{\mathrm{sm},v,i}$ is smooth. A companion point $(\rho,\delta')$ of $x$ falls in one of the following two types: 
\begin{enumerate}[label=(\alph*)]
	\item $\underline{\delta}_{\mathrm{alg}}'\neq \underline{\delta}_{\mathrm{alg}}$ but $\underline{\delta}_{\mathrm{sm}}'=\underline{\delta}_{\mathrm{sm}}$ (different ``weights'');
	\item $\underline{\delta}_{\mathrm{sm}}'\neq \underline{\delta}_{\mathrm{sm}}$ (different ``$U_p$-eigenvalues up to some normalizations'').
\end{enumerate}
Our main theorem is the following.
\begin{theorem}[Theorem \ref{theoremmaincrystalline}]\label{thm:introduction}
	Suppose that $x=(\rho,\underline{\delta})\in Y(U^p,\overline{\rho})$ is a point such that $\rho_v$ is generic crystalline (see \S\ref{sec:companionpointsdesciption} for the generic condition) for all $v\in S_p$. Assume the tame level is sufficiently small and the usual Taylor-Wiles hypothesis (Assmption \ref{ass:taylorwiles}). Then all the companion points of $x$ in the conjecture of Hansen or Breuil appear on $Y(U^p,\overline{\rho})$.
\end{theorem}
The above theorem was already proved by Breuil-Hellmann-Schraen in \cite{breuil2019local} under the assumption that the Hodge-Tate weights of each $\rho_v$ are regular (i.e. pairwise different). In \cite{wu2021local}, the author removed the regular assumption on the Hodge-Tate weights, but only was able to prove the existence of all companion points of the type (a) above in the non-regular cases. The task of this paper is to find all companion points of type (b) for non-regular points. These are companion points corresponding to different triangulation (refinements) of the trianguline (crystalline) Galois representations.\par
The proof of our theorem is motivated by some arguments in ordinary cases and will use the known results in both regular and non-regular cases. In ordinary cases, modularity lifting theorems were proved for ordinary families of Galois representations that will specialize to companion points with possibly non-regular weights. In our finite slope/trianguline cases, for a non-regular point $x$ as in Theorem \ref{thm:introduction}, the naive strategy is to find a sequence of points $x^i$ on $Y(U^p,\overline{\rho})$ with regular Hodge-Tate weights such that $x=\varinjlim_{i}x^i$ and certain companion points $(x^i)'$ of $x^i$, which will exist on $Y(U^p,\overline{\rho})$ using \cite{breuil2019local}, satisfy that $(x^i)'$ converge to a point $x'$ on $Y(U^p,\overline{\rho})$ and that $x'$ is a companion point of $x$ of type (b). \par
The actual proof is Galois-theoretical. Using patching methods \cite{caraiani2016patching} and the patched eigenvariety \cite{breuil2017interpretation}, we can reduce the task to find those nearby regular $x^i$ to a similar problem on the \emph{trianguline variety} in \cite{breuil2017interpretation}, the local Galois-theoretical eigenvariety. Those $x^i=(\rho^i,\underline{\delta}^i)$ (now $\rho^i$ are representations of local Galois groups) are found by studying some ``crystalline/de Rham'' loci on the moduli space of trianguline $(\varphi,\Gamma)$-modules in the proof of \cite[Thm. 2.6]{breuil2017interpretation} and $\rho^i$ will be the Galois representations corresponding to certain \'etale trianguline $(\varphi,\Gamma)$-modules of parameter $\underline{\delta}^i$ (the \'etaleness will be achieved by the results of Hellmann in \cite{hellmann2016families}). The key example is the case $n=2$.
\begin{remark}
	Our proof for the existence of companion points of type (b) will not use directly the theory of local models of the trianguline variety in \cite{breuil2019local} and \cite{wu2021local}. However, it is the existence of all companion points of type (a) in \cite{wu2021local}, which used the local models, that allows us to keep working in the smooth locus of the trianguline variety consisting of points $(\rho,\underline{\delta})$ where $\rho$ is trianguline of parameter $\underline{\delta}$.	
\end{remark}
\begin{remark}
	As a corollary of Theorem \ref{thm:introduction}, we can determine all the companion constituents (certain locally analytic representations of $\prod_{v\in S_p}\GL_n(F_{v}^+)$) in the Hecke-isotypic part of the completed cohomology of $G$ associated with generic crystalline Galois representations in Breuil's locally analytic socle conjecture (Corollary \ref{cor:socle}). Since there are no locally algebraic constituents in the non-regular cases, the existence of all of these companion constituents could be a replacement in the automorphic side to compare with de Rhamness of the Galois side in the $p$-adic Langlands correspondence in this particular non-regular situation.
\end{remark}
\begin{remark}
	For non-regular weights, the existence of all companion points will not lead directly to a classicality result of Hecke eigensystems in contrast with \cite{buzzard1999companion} since classical automorphic representations for definite unitary groups will have regular weights. The results of this paper might be able to be adapted for Hilbert modular forms and have applications in classicality of $p$-adic Hilbert modular forms with non-regular and possibly non-parallel weights.
\end{remark}
The paper is organized as follows. In \S\ref{sec:coho} and \S\ref{sec:crystallin}, we collect some (presumedly known) results on $(\varphi,\Gamma)$-modules over the Robba rings. In \S\ref{sec:hunting}, we find the companion points on the trianguline variety. In \S\ref{sec:global}, we apply the local results in \S\ref{sec:hunting} to the global settings and prove the main theorem.
\subsection{Notation}
We will use the notation in \cite[\S1.7]{wu2021local}. Let $K$ be a finite extension of $\Q_p$ and $L/\Q_p$ be a large enough coefficient field such that $\Sigma:=\{\tau:K\hookrightarrow L\}$ has size $[K:\Q_p]$. Let $C$ be the completion of an algebraic closure of $K$. We have the Robba ring $\cR_{L,K}$ of $K$ over $L$ defined in \cite[Def. 6.2.1]{kedlaya2014cohomology}. Let $t\in \cR_{L,K}$ denote Fontaine's $2\pi i$ and $t=u\prod_{\tau \in \Sigma}t_{\tau}$ for some $u\in \cR_{L,K}^{\times}$ (see \cite[Not. 6.2.7]{kedlaya2014cohomology} for details). For $\mathbf{k}=(k_{\tau})_{\tau\in\Sigma}\in\Z^{\Sigma}$, write $t^{\mathbf{k}}=\prod_{\tau\in\Sigma}t_{\tau}^{k_{\tau}}$. If $\delta:K^{\times}\rightarrow L^{\times}$ is a continuous character, let $\cR_{L,K}(\delta)$ be the associated rank one $(\varphi,\Gamma_K)$-modules over $\cR_{L,K}$ in \cite[Cons. 6.2.4]{kedlaya2014cohomology} where $\Gamma_K=\Gal(K(\mu_{\infty})/K)$. Then $t^{\mathbf{k}}\cR_{L,K}=\cR_{L,K}(z^{\mathbf{k}})$ where $z^{\mathbf{k}}$ denotes the character $z\mapsto \prod_{\tau\in\Sigma}\tau(z)^{k_{\tau}}$. If $a\in L^{\times}$, then denote by $\mathrm{unr}(a)$ the unramified character of $K^{\times}$ sending a uniformizer of $K$ to $a$. Let $\cT_{L}$ be the rigid space over $L$ parametrizing continuous characters of $K^{\times}$ and $\cT_{0}\subset \cT_{L}$ be the complement of the subset of characters $\delta$ such that $\delta$ or $\epsilon \delta^{-1}$ is algebraic. Here $\epsilon$ is the character $\mathrm{Norm}_{K/\Q_p}|\mathrm{Norm}_{K/\Q_p}|_{\Q_p}$ of $K^{\times}$. We can define $\tau$-part $\wt_{\tau}(\delta)$ of the weight $\wt(\delta)$ of $\delta$ (see \cite[\S1.7.2]{wu2021local}). The cyclotomic character of $\cG_K:=\Gal(\overline{K}/K)$ has Hodge-Tate weights one. We fix an integer $n\geq 2$. 
\section{Cohomology of $(\varphi,\Gamma_K)$-modules}\label{sec:coho}
We collect some results of the cohomology of $(\varphi,\Gamma_K)$-modules (of character type). We fix $\delta:K^{\times}\rightarrow L^{\times}$ to be a continuous character. Recall if $D$ is a $(\varphi,\Gamma_K)$-module over $\cR_{L,K}$, then $H^{i}_{\varphi,\gamma_K}(D[\frac{1}{t}])=\varinjlim_{m\to +\infty}H^{i}_{\varphi,\gamma_K}(t^{-m}D), i=0,1,2$ (\cite[(3.11)]{breuil2019local}).
\begin{lemma}\label{lem:dimensionH1} 
	If $\delta\in \cT_0$, then $\dim_LH^{i}_{\varphi,\gamma_K}(\cR_{L,K}(\delta))=0$ for $i=0,2$ and $\dim_LH^{1}_{\varphi,\gamma_K}(\cR_{L,K}(\delta))=[K:\Q_p]$.
\end{lemma}
\begin{proof}
	\cite[Prop. 6.2.8]{kedlaya2014cohomology}.
\end{proof}
Recall in \cite[\S3.3]{breuil2019local} we have a functor $W_{\mathrm{dR}}$ (resp. $W_{\mathrm{dR}}^+$) sending a $(\varphi,\Gamma_K)$-module over $\cR_{L,K}[\frac{1}{t}]$ (resp. $\cR_{L,K}$) to an $L\otimes_{\Q_p}\mathrm{B}_{\mathrm{dR}}$-representations (resp. $L\otimes_{\Q_p}\mathrm{B}_{\mathrm{dR}}^+$-representation) of $\cG_K$.
\begin{lemma}\label{lem:cohomologyderhamphigamma}
	If $\delta\in \cT_0$ and is locally algebraic, then
	\[H^{1}_{\varphi,\gamma_K}(\cR_{L,K}(\delta)[\frac{1}{t}])\simrightarrow H^1(\cG_K, W_{\mathrm{dR}}(\cR_{L,K}(\delta)[\frac{1}{t}])).\]
\end{lemma}
\begin{proof}
	\cite[Lem. 3.4.2]{breuil2019local}
\end{proof}
\begin{proposition}
	For $\mathbf{k}\in \Z_{\geq 0}^{\Sigma},i=0,1$, 
	\[\dim_L H^{i}(t^{-\mathbf{k}}\cR_{L,K}(\delta)/\cR_{L,K}(\delta))=|\{\tau\in\Sigma\mid k_{\tau}\geq 1, \wt_{\tau}(\delta)\in\{1,\cdots,k_{\tau} \} \}|. \]
\end{proposition}
\begin{proof}
	This follows from \cite[Appendix A]{kedlaya2009some} and \cite[Lem. 2.16]{nakamura2009classification} (and some other well-known results: \cite[Thm.4.7, Cor. 4.8]{liu2007cohomology} and the comparison in \cite[Prop. 2.2]{nakamura2009classification}, or a generalized version \cite[Thm. 5.11]{nakamura2014deformations}).
\end{proof}
\begin{corollary}\label{cor:dimensionker}
	For $\mathbf{k}\in \Z_{\geq 0}^{\Sigma}$, and $\delta\in\cT_0$, then 
	\begin{align*}
		&\dim_L\mathrm{Ker}\left(H^{1}_{\varphi,\gamma_K}(\cR_{L,K}(\delta))\rightarrow H^{1}_{\varphi,\gamma_K}(t^{-\mathbf{k}}\cR_{L,K}(\delta))\right)\\
		=&\dim_L\mathrm{Coker}\left(H^{1}_{\varphi,\gamma_K}(\cR_{L,K}(\delta))\rightarrow H^{1}_{\varphi,\gamma_K}(t^{-\mathbf{k}}\cR_{L,K}(\delta))\right)\\
		=&|\{\tau\in\Sigma\mid k_{\tau}\geq 1, \wt_{\tau}(\delta)\in\{1_{\tau},\cdots,k_{\tau}\} \}|.
	\end{align*}
\end{corollary}
\begin{corollary}\label{cor:cohomologyiso}
	If $\delta\in \cT_0$ is locally algebraic and $\wt_{\tau}(\delta)\leq 0$ for all $\tau\in\Sigma$, then the natural maps $\cR_{L,K}(\delta)\hookrightarrow t^{-\mathbf{k}}\cR_{L,K}(\delta)$ induce isomorphisms $H^{i}_{\varphi,\gamma_K}(\cR_{L,K}(\delta))\simrightarrow H^{i}_{\varphi,\gamma_K}(t^{-\mathbf{k}}\cR_{L,K}(\delta))$ for all $i=0,1,2, \mathbf{k}\in \Z_{\geq 0}^{\Sigma}$.
\end{corollary}
If $\mathbf{k}\in\Z^{\Sigma}$, write $\mathbf{k}^{\sharp}\in \Z^{\Sigma}$ where $k_{\tau}^{\sharp}=k_{\tau}$ if $k_{\tau}\geq 1$ and $k_{\tau}^{\sharp}=0$ otherwise. 
\begin{proposition}\label{prop:kernelandderham}
	Assume that $\delta\in \cT_0$ is locally algebraic with weights $\mathbf{k}\in \Z^{\Sigma}$. Then the image of an element $x\in H^{1}_{\varphi,\gamma_K}(\cR_{L,K}(\delta))$ in $H^1_{\varphi,\gamma_K}(\cR_{L,K}(\delta)[\frac{1}{t}])$ is $0$ if and only if
	\[x\in\mathrm{Ker}\left(H^{1}_{\varphi,\gamma_K}(\cR_{L,K}(\delta))\rightarrow H^{1}_{\varphi,\gamma_K}(t^{-\mathbf{k}^{\sharp}}\cR_{L,K}(\delta))\right). \]
\end{proposition}
\begin{proof}
	We have $\wt_{\tau}(\delta z^{-\mathbf{k}^{\sharp}})\leq 0$ for all $\tau\in \Sigma$. Thus by Corollary \ref{cor:cohomologyiso}, $H^{1}_{\varphi,\gamma_K}(t^{-\mathbf{k}^{\sharp}}\cR_{L,K}(\delta))\simeq H^{1}_{\varphi,\gamma_K}(\cR_{L,K}(\delta)[\frac{1}{t}])$. 
\end{proof}
\section{A crystalline criterion}\label{sec:crystallin}
We need some criterion to guarantee that the points on the trianguline variety we will find in the next section are crystalline. For the definition of de Rham or crystalline $(\varphi,\Gamma_K)$-modules, see \cite[Def. 2.5]{hellmann2016density}. We say a trianguline $(\varphi,\Gamma_K)$-module of parameter $\underline{\delta}=(\delta_1,\cdots,\delta_n)$ is generic if $\delta_i\delta_j^{-1}\in\cT_0$ for all $i\neq j$ (or $\underline{\delta}\in\cT_0^n$ in the notation of \cite[\S3.2]{wu2021local}, remark that $\cT_0^n\neq (\cT_0)^n$!). Recall that a locally algebraic character $\delta:K^{\times}\rightarrow L^{\times}$ is crystalline (or semi-stable) if and only if the smooth part $\delta_{\mathrm{sm}}$ is unramified (see \cite[Exam. 6.2.6]{kedlaya2014cohomology}).
\begin{lemma}\label{lem:crystalline}
	If $D$ is a generic trianguline $(\varphi,\Gamma_K)$-module over $\cR_{L,K}$ of parameter $\underline{\delta}$ such that all $\delta_i$ are crystalline, then $D$ is a crystalline $(\varphi,\Gamma_K)$-module if and only if $D$ is de Rham.
\end{lemma}
\begin{proof}
	This follows from the proof of \cite[Cor. 2.7(i)]{hellmann2016density}. Assume $D$ is de Rham. As $D$ is a successive extension of crystalline $(\varphi,\Gamma_K)$-modules, $D$ is semi-stable (by \cite{berger2008equations}, see also the arguments in \cite[\S6.1]{berger2002representations}). By the generic assumption, the monodromy must be trivial. Hence $D$ is crystalline.
\end{proof}
\begin{lemma}
	Let $D$ be a trianguline $(\varphi,\Gamma_K)$-module over $\cR_{L,K}$ of rank $n$ with the trianguline filtration $\Fil_{\bullet}D$ such that $\Fil_{i}D/\Fil_{i-1}D\simeq \cR_{L,K}(\delta_i)$ for $i=1,\cdots,n$. Fix $i_0\in\{1,\cdots,n-1\}$ and let $D_0=\Fil_{i_0}D$ and $D_1=D/\Fil_{i_0}D$. Assume that $\underline{\delta}$ is locally algebraic and let $\lambda=(\lambda_{\tau,i})_{\tau\in\Sigma,i=1,\cdots,n}=\wt(\underline{\delta})\in (\Z^{\Sigma})^n$. Assume that for every $\tau\in\Sigma$, $\lambda_{\tau,i}< \lambda_{\tau,j}$ if $i>i_0\geq j$ and that both $D_0$ and $D_1$ are de Rham, then $D$ is de Rham.
\end{lemma}
\begin{proof}
	This is a generalization of \cite[Prop. 2.6]{hellmann2016density}.
	We need to prove that $\dim_LW_{\mathrm{dR}}(D)^{\cG_K}=n[K:\Q_p]$. For $\tau\in \Sigma$, let $k_{\tau}=\mathrm{max}_{i>i_0}\lambda_{\tau,i}$. Then $\dim_LW_{\mathrm{dR}}^+(t^{-\mathbf{k}}D_0)^{\cG_K}=0$ as the Hodge-Tate weights of $t^{-\mathbf{k}}D_0$ are positive and $t^{-\mathbf{k}}D_0$ is de Rham. We have an exact sequence
	\[0\rightarrow W_{\mathrm{dR}}^+(t^{-\mathbf{k}}D)^{\cG_K}\rightarrow W_{\mathrm{dR}}^+(t^{-\mathbf{k}}D_1)^{\cG_K}\rightarrow H^1(\cG_K,W_{\mathrm{dR}}^+(t^{-\mathbf{k}}D_0)).\]
	The Hodge-Tate weights of $t^{-\mathbf{k}}D_0$ are $\geq 1$, hence $H^1(\cG_K,W_{\mathrm{dR}}^+(t^{-\mathbf{k}}D_0))=0$ by \cite[Cor. 5.6]{nakamura2014deformations} (we have that $H^1(\cG_K,C(i))=0$ for $i\neq 0$ by \cite[Prop. 2.15(ii)]{fontaine2004arithmetique}). We get $\dim_LW_{\mathrm{dR}}^+(t^{-\mathbf{k}}D)^{\cG_K}=\dim_LW_{\mathrm{dR}}^+(t^{-\mathbf{k}}D_1)^{\cG_K}=(n-i_0)[K:\Q_p]$ since $t^{-\mathbf{k}}D_1$ is de Rham with non-positive Hodge-Tate weights. As $D_0$ is de Rham, $\dim_LW_{\mathrm{dR}}(D_0)^{\cG_K}=i_0[K:\Q_p]$. Since $W_{\mathrm{dR}}^+(t^{-\mathbf{k}}D_0)[\frac{1}{t}]\cap W_{\mathrm{dR}}^+(t^{-\mathbf{k}}D)=W_{\mathrm{dR}}^+(t^{-\mathbf{k}}D_0)$, we have $W_{\mathrm{dR}}(D_0)^{\cG_K}\cap W_{\mathrm{dR}}^+(t^{-\mathbf{k}}D)^{\cG_K}=W_{\mathrm{dR}}^+(t^{-\mathbf{k}}D_0)^{\cG_K}=\{0\}$. Then $W_{\mathrm{dR}}^+(t^{-\mathbf{k}}D)^{\cG_K}$ and $W_{\mathrm{dR}}(D_0)^{\cG_K}$ span an $n[K:\Q_p]$-dimensional $L$-subspace in $W_{\mathrm{dR}}(D)^{\cG_K}$.
\end{proof}
The above lemma will be used in the following form later.
\begin{proposition}\label{prop:de Rham}
	Assume that $D$ is a trianguline $(\varphi,\Gamma_K)$-module of rank $n$ over $\cR_{L,K}$ with the trianguline filtration $\Fil_{\bullet}D$ such that $\Fil_{i}D/\Fil_{i-1}D\simeq \cR_{L,K}(\delta_i)$ for $i=1,\cdots,n$. Fix $i_0\in\{1,\cdots,n-1\}$ and let $D_0=\Fil_{i_0-1}D, D_1=\Fil_{i_0+1}D/D_0$ and $D_2=D/\Fil_{i_0+1}D$. Let $\lambda=\wt(\underline{\delta})$ and assume that $\underline{\delta}$ is locally algebraic. Assume that for every $\tau\in\Sigma$, $\lambda_{\tau,i}> \lambda_{\tau,i+1}$ if $i\neq i_0$, $\lambda_{\tau,i}> \lambda_{\tau,i_0},\lambda_{\tau,i_0+1}$ if $i<i_0$, and $\lambda_{\tau,i}<\lambda_{\tau,i_0},\lambda_{\tau,i_0+1}$ if $i>i_0+1$. If $D_1$ is de Rham, then $D$ is de Rham.
\end{proposition}
%\begin{remark}
%	There is a ``direct'' proof for the above facts in some cases. We will use notation in \cite[\S2, \S3]{wu2021local}. Assume that $D=D_{\mathrm{rig}}(r)$ where $x=(r,\underline{\delta})\in U_{\mathrm{tri}}(\overline{r})$ and $\lambda=\wt(\underline{\delta})$ satisfies the requirements in the above proposition. We assume that $\lambda_{\tau,i}\neq \lambda_{\tau,j}$ for all $\tau\in\Sigma$ and $i\neq j$ for simplicity. Let $G=\prod_{\tau\in\Sigma} \GL_{n}$ and $Q=\prod_{\tau\in\Sigma} B(\alpha_{i_0})$ be the standard parabolic subgroup, where $B(\alpha_{i_0})$ denotes the minimal Borel subgroup of $\GL_n$ whose standard Leve subgroup has the form $(\GL_1)^{i_0-1}\times\GL_2\times (\GL_1)^{n-i_0-1}$, such that $(D,\Fil_{\bullet}D)$ is $Q$-de Rham (\cite[Def. 3.10]{wu2021local}). Write $\mathbf{h}$ for the antidominant Hodge-Tate weights of $r$. Then by \cite[Cor. 2.20, Thm. 2.24, Thm. 3.8]{wu2021local}, there exist $w\in W_G$, $x_{\mathrm{pdR}}\in Z_{B,w}\subset Z_{Q,B}$ and $\lambda'=w(\mathbf{h})$ such that $\lambda-\rho$ is strongly linked to $\lambda'-\rho$ (\cite[\S5.1]{humphreys2012introduction}, in particular $\lambda'-\lambda\geq 0$) and such that $\lambda'$ is strictly $Q$-dominant. Then it is easy to see that $\lambda'$ must be dominant. Hence $x_{\mathrm{pdR}}\in Z_{B,w_0}$ which implies that $x$ is de Rham.
%\end{remark}
\section{Critical points hunting}\label{sec:hunting}
Let $\overline{r}:\cG_K\rightarrow \GL_n(k_L)$ be a continuous representation. We firstly recall some constructions around the trianguline variety $X_{\mathrm{tri}}(\overline{r})$ in \cite[\S2.2]{breuil2017interpretation}. In the first parts of this section, we will only need the Zariski open dense subset $U_{\mathrm{tri}}(\overline{r})\subset \trivar$. 
\subsection{The trianguline variety}\label{subsec:trianguline}
Let $\cT_{\mathrm{reg}}^n$ be the Zariski open subset of $\cT_L^n$ consisting of characters $\underline{\delta}=(\delta_i)_{i=1,\cdots,n}$ such that $\delta_i\delta_j^{-1}\neq z^{-\mathbf{k}},\epsilon z^{\mathbf{k}}$ for $i\neq j$ and $\mathbf{k}\in\Z_{\geq 0}^{\Sigma}$. There are rigid spaces $\mathcal{S}_n^{\square}(\overline{r})\rightarrow \mathcal{S}_n$ over $\cT^n_{\mathrm{reg}}$ in the proof of \cite[Thm. 2.6]{breuil2017interpretation} (and in \cite[\S2.2]{hellmann2016density}) which will be used later, and we recall below. \par
The space $\mathcal{S}_n$ represents the functor sending a reduced rigid space $X$ over $L$ to the isomorphic classes of quadruples $(D_X,\Fil_{\bullet}D_X,\nu_X, \underline{\delta}_X)$ where $D_X$ is a $(\varphi,\Gamma_K)$-module of rank $n$ over $\cR_{X,K}$ where $\cR_{X,K}$ denotes the Robba ring of $K$ over $X$, $\Fil_{\bullet}D_X$ is a filtration of sub-$(\varphi,\Gamma_K)$-modules of $D_X$ which are locally direct summands as $\cR_{X,K}$-modules, $\underline{\delta}_X\in \cT_{\mathrm{reg}}^n(X)$ and $\nu_X: \Fil_{i}D_X/\Fil_{i-1}D_X\simeq \cR_{X,K}(\delta_{i})$ (we omit the subscripts of the spaces for the universal characters to simplify the notation). There are obvious morphisms $\mathcal{S}_n\rightarrow \mathcal{S}_{n-1}\times_L\cT_L\rightarrow \cT_{\mathrm{reg}}^{n-1}\times_L \cT_L\subset \cT_L^n$. Let $U\subset \mathcal{S}_{n-1}\times_L\cT_L$ be the preimage of $\cT_{\mathrm{reg}}^n$ which is Zariski open in ${S}_{n-1}\times_L \cT_L$ and let $D_U$ be the $(\varphi,\Gamma)$-modules over $U$ pulled back from the universal one on $\mathcal{S}_{n-1}$. Then $\mathcal{S}_n\simeq \mathrm{Spec}^{\mathrm{an}}(\mathrm{Sym}^{\bullet}(\mathcal{E}xt^1_{\varphi,\gamma_K}(\cR_{U,K}(\delta_n), D_U)^{\vee}))$ is a geometric vector bundle over $U$ where $\mathcal{E}xt^1_{\varphi,\gamma_K}(\cR_{U,K}(\delta_n), D_U)\simeq H^1_{\varphi,\gamma_K}(D_U(\delta_n^{-1}))$ is a locally free sheaf on $U$ of rank $(n-1)[K:\Q]$ (\cite[Prop. 2.3]{hellmann2016density}) and the notion $\mathrm{Spec}^{\mathrm{an}}$ is taken from \cite[Thm. 2.2.5]{conrad2006relative}. It follows from induction that the map $\mathcal{S}_n\rightarrow \cT^n_{\mathrm{reg}}\subset \cT_L^n$ is smooth. \par
Let $\mathcal{S}_n^{\mathrm{adm}}\subset\mathcal{S}_n$ be the open subset (as adic spaces) of the admissible locus, which comes from a rigid space, and let $\mathcal{S}_n^{\square,\mathrm{adm}}\rightarrow \mathcal{S}_n^{\mathrm{adm}}$ be the $\GL_n$-torsor trivializing the universal Galois representation over $\mathcal{S}^{\mathrm{adm}}_n$. Let $\mathcal{S}_n^{\square,\mathrm{adm},+}\subset\mathcal{S}_n^{\square,\mathrm{adm}}$ be the admissible open subset where the universal framed representation $\cG_K\rightarrow \GL_n(\Gamma(\mathcal{S}_n^{\square,\mathrm{adm}},\cO_{\mathcal{S}_n^{\square,\mathrm{adm}}}))$ factors through $\cG_K\rightarrow \GL_n(\Gamma(\mathcal{S}_n^{\square,\mathrm{adm}},\cO_{\mathcal{S}_n^{\square,\mathrm{adm}}}^+))$. We denote by $\mathcal{S}_n^{\square}(\overline{r})$ the admissible open subset of $\mathcal{S}_n^{\square,\mathrm{adm},+}$ where the reduction 
$\cG_K\rightarrow {\GL_n(\Gamma(\mathcal{S}_n^{\square,\mathrm{adm}},\cO_{\mathcal{S}_n^{\square,\mathrm{adm},+}}^+/\cO_{\mathcal{S}_n^{\square,\mathrm{adm},+}}^{++}))}$ coincides with $\overline{r}$ 
(see also the discussion before \cite[Prop. 8.17]{hartl2020universal}). The map $\kappa:\mathcal{S}_n^{\square}(\overline{r})\rightarrow \mathcal{S}_n\rightarrow \cT_L^n$ is also smooth. \par
Let $R_{\overline{r}}$ (over $\cO_L$) be the framed deformation ring of $\overline{r}$ and let $\fX_{\overline{r}}:=\Spf(R_{\overline{r}})^{\mathrm{rig}}$ be the rigid generic fiber (we follow the notation in \cite{breuil2019local} and \cite{wu2021local} rather than \cite{breuil2017interpretation}). The image of $\mathcal{S}_n^{\square}(\overline{r})\rightarrow \fX_{\overline{r}}\times\cT_L^n$ is equal to $U_{\mathrm{tri}}(\overline{r})$ and the trianguline variety $X_{\mathrm{tri}}(\overline{r})$ is the Zariski closure of $U_{\mathrm{rig}}(\overline{r})$ in $\fX_{\overline{r}}\times\cT_L^n$ with the reduced induced structure. The map $\pi_{\overline{r}}:\mathcal{S}_n^{\square}(\overline{r})\rightarrow U_{\mathrm{tri}}(\overline{r})\subset X_{\mathrm{tri}}(\overline{r})$ is smooth. 
\subsection{Some ``de Rham'' locus}
We will define some subspace $\mathcal{S}^{\square}_{n,(i_0,\mathbf{k}_J)}(\overline{r})\subset \mathcal{S}^{\square}_{n}(\overline{r})$ where the criteria in the last section will apply for certain points on it. \par
We fix datum $i_0\in\{1,\cdots,n-1\}$, a subset $J\subset \Sigma$ and $\mathbf{k}_J=(k_{\tau})_{\tau\in J}\in \Z_{\geq 1}^J$. We allow $J$ to be $\emptyset$ or $\Sigma$. Let $\mathcal{T}_{(i_0,\mathbf{k}_J)}^n$ be the subset of characters $\underline{\delta}\in\cT_{\mathrm{reg}}^n$ such that $\mathrm{wt}_{\tau}(\delta_{i_0}\delta_{i_0+1}^{-1})=k_{\tau}$ for all $\tau\in J$ and $\delta_{i_0}\delta_{i_0+1}^{-1}\in \cT_0$. Let $\ft$ be the base change to $L$ of the $\Q_p$-Lie algebra of $(K^{\times})^n$ and view its dual $\ft^{*}$ as the affine space of weights and we have a weight map $\wt:\cT_L^n\rightarrow \ft^{*}$. Let $\ft^{*}_{(i_0,\mathbf{k}_J)}$ be the subspace of points $(\lambda_{\tau,i})_{\tau\in\Sigma,i=1,\cdots,n}$ such that $\lambda_{\tau,i_0}-\lambda_{\tau,i_0+1}=k_{\tau}$ for all $\tau\in J$.
\begin{lemma}
	The rigid space $\mathcal{T}_{(i_0,\mathbf{k}_J)}^n$ is smooth reduced equidimensional and is \'etale over $\ft^*_{(i_0,\mathbf{k}_J)}$.
\end{lemma}
\begin{proof}
	This follows from \cite[Prop. 6.1.13]{ding2017formes}.
\end{proof}
Consider the universal $(\varphi,\Gamma_K)$-modules $D_X$ and $\Fil_{\bullet}D_X$ over $X=\mathcal{S}_{n}\times_{\cT_L^n}\cT^n_{(i_0,\mathbf{k}_J)}$ or $ \mathcal{S}_{n}^{\square}(\overline{r})\times_{\cT_L^n}\cT^n_{(i_0,\mathbf{k}_J)}$ pulled back from $\mathcal{S}_n$. The extension
\[0\rightarrow \cR_{X,K}(\delta_{X,i_0})\rightarrow \Fil_{i_0+1}D_X/\Fil_{i_0-1}D_X\rightarrow \cR_{X,K}(\delta_{X,i_0+1})\rightarrow 0 \]
together with the trivialization $\nu_X$ defines a section $s_X$ in 
\[\mathcal{E}xt^1_{\varphi,\gamma_K}(\cR_{X,K}(\delta_{X,i_0+1}),\cR_{X,K}(\delta_{X,i_0}))\simeq H^1_{\varphi,\gamma_K}(\cR_{X,K}(\delta_{X,i_0}\delta_{X,i_0+1}^{-1})).\]
By the main result of \cite{kedlaya2014cohomology}, both $H^1_{\varphi,\gamma_K}(\cR_{X,K}(\delta_{i_0}\delta_{i_0+1}^{-1}))$ and $H^1_{\varphi,\gamma_K}(t^{-\mathbf{k}_J}\cR_{X,K}(\delta_{i_0}\delta_{i_0+1}^{-1}))$ are coherent sheaves on $X$. We define the subspace $\mathcal{S}_{n,(i_0,\mathbf{k}_J)}$ or $\mathcal{S}_{n,(i_0,\mathbf{k}_J)}^{\square}(\overline{r})$ to be the vanishing locus on $X=\mathcal{S}_{n}\times_{\cT_L^n}\cT^n_{(i_0,\mathbf{k}_J)}$ or $\mathcal{S}_{n}^{\square}(\overline{r})\times_{\cT_L^n}\cT^n_{(i_0,\mathbf{k}_J)}$ of the image of $s_X$ under the natural map 
\[H^1_{\varphi,\gamma_K}(\cR_{X,K}(\delta_{i_0}\delta_{i_0+1}^{-1}))\rightarrow H^1_{\varphi,\gamma_K}(t^{-\mathbf{k}_J}\cR_{X,K}(\delta_{i_0}\delta_{i_0+1}^{-1})).\] 
The vanishing loci are Zariski closed subspaces as $H^1_{\varphi,\gamma_K}(t^{-\mathbf{k}_J}\cR_{X,K}(\delta_{i_0}\delta_{i_0+1}^{-1}))$ is locally free.
\begin{lemma}\label{lem:locallyfree}
	Let $X$ be a reduced rigid space over $L$ and $\delta_X:K^{\times}\rightarrow \Gamma(X,\cO_X)^{\times}$ be a continuous character. Assume that for any $x\in X$, we have $\delta_x\in\mathcal{T}_0$ and $\wt_{\tau}(\delta_x)=k_{\tau}$ for all $\tau\in J$. Then the coherent sheaves $H^1_{\varphi,\gamma_K}(\cR_{X,K}(\delta_X))$, $H^1_{\varphi,\gamma_K}(t^{-\mathbf{k}_J}\cR_{X,K}(\delta_X))$, as well as 
	\[\mathrm{Ker}(H^1_{\varphi,\gamma_K}(\cR_{X,K}(\delta_X))\rightarrow H^1_{\varphi,\gamma_K}(t^{-\mathbf{k}_J}\cR_{X,K}(\delta_X)))\]
	and 
	\[\mathrm{Coker}(H^1_{\varphi,\gamma_K}(\cR_{X,K}(\delta_X))\rightarrow H^1_{\varphi,\gamma_K}(t^{-\mathbf{k}_J}\cR_{X,K}(\delta_X)))\]
	are finite projective over $X$ of rank $|\Sigma|,|\Sigma|, |J|,|J|$ respectively and their formation commutes with arbitrary base change.
\end{lemma}
\begin{proof}
	We write $\mathrm{Ker}(\delta_X)$ or $\mathrm{Coker}(\delta_X)$ for the kernel or the cokernel of the map
	\[H^1_{\varphi,\gamma_K}(\cR_{X,K}(\delta_X))\rightarrow H^1_{\varphi,\gamma_K}(t^{-\mathbf{k}_J}\cR_{X,K}(\delta_X))\]
	for simplicity.\par
	For any $x\in X$, $\dim_{k(x)}H^1_{\varphi,\gamma_K}(\cR_{k(x),K}(\delta_x))=\dim_{k(x)}H^1_{\varphi,\gamma_K}(t^{-\mathbf{k}_J}\cR_{k(x),K}(\delta_x))=|\Sigma|$ by Lemma \ref{lem:dimensionH1} and $\dim_{k(x)}\mathrm{Ker}(\delta_x)=\dim_{k(x)}\mathrm{Coker}(\delta_x)=|J|$ by Corollary \ref{cor:dimensionker} and our assumptions on $\delta_x$. The fact that $H^1_{\varphi,\gamma_K}(\cR_{k(x),K}(\delta_x))$ and $H^1_{\varphi,\gamma_K}(t^{-\mathbf{k}_J}\cR_{k(x),K}(\delta_x))$ are locally free and commute with base change of the form $\Sp(k(x))\rightarrow X$ for $x\in X$ follows from \cite[Prop. 2.3]{hellmann2016density}. Thus for any $x\in X$, 
	\begin{align*}
		&\mathrm{Coker}(\delta_{X})\otimes_{\cO_X}k(x)\\
		\simeq&\mathrm{Coker}(H^1_{\varphi,\gamma_K}(\cR_{X,K}(\delta_X))\otimes_{\cO_X}k(x)\rightarrow H^1_{\varphi,\gamma_K}(t^{-\mathbf{k}_J}\cR_{X,K}(\delta_X))\otimes_{\cO_X}k(x))\\
		\simeq&\mathrm{Coker}(H^1_{\varphi,\gamma_K}(\cR_{k(x),K}(\delta_x))\rightarrow H^1_{\varphi,\gamma_K}(t^{-\mathbf{k}_J}\cR_{k(x),K}(\delta_x)))\\
		\simeq &\mathrm{Coker}(\delta_x).
	\end{align*}
	Thus $\mathrm{Coker}(\delta_{X})$ has constant rank, hence projective by \cite[Lem. 2.1.8 (1)]{kedlaya2014cohomology}, and commutes with base change of the form $\Sp(k(x))\rightarrow X$ for $x\in X$.\par
	Let $\mathrm{Im}(\delta_X)$ be the image of the map $H^1_{\varphi,\gamma_K}(\cR_{X,K}(\delta_X))\rightarrow H^1_{\varphi,\gamma_K}(t^{-\mathbf{k}_J}\cR_{X,K}(\delta_X))$. Then we have $\dim_{k(x)}\mathrm{Im}(\delta_x)=|\Sigma|-|J|$ for any $x\in X$. By the exact sequence
	\[0\rightarrow \mathrm{Im}(\delta_X)\rightarrow H^1_{\varphi,\gamma_K}(t^{-\mathbf{k}_J}\cR_{X,K}(\delta_X))\rightarrow \mathrm{Coker}(\delta_X)\rightarrow 0\]
	and $\Tor^1_{\cO_X}(k(x),\mathrm{Coker}(\delta_X))=0$, we get 
	\[0\rightarrow \mathrm{Im}(\delta_X)\otimes_{\cO_X}k(x)\rightarrow H^1_{\varphi,\gamma_K}(t^{-\mathbf{k}_J}\cR_{k(x),K}(\delta_x))\rightarrow \mathrm{Coker}(\delta_x)\rightarrow 0\]
	for any $x\in X$. Hence $\mathrm{Im}(\delta_X)\otimes_{\cO_X}k(x)\simeq \mathrm{Im}(\delta_x)$ for any $x\in X$ and $\mathrm{Im}(\delta_X)$ is finite projective of rank $|\Sigma|-|J|$. Repeat the argument using the exact sequence 
	\[0\rightarrow \mathrm{Ker}(\delta_X)\rightarrow H^1_{\varphi,\gamma_K}(\cR_{X,L}(\delta_X))\rightarrow \mathrm{Im}(\delta_X)\rightarrow 0\] 
	and that $\Tor^1_{\cO_X}(k(x),\mathrm{Im}(\delta_X))=0$, we see $\mathrm{Ker}(\delta_X)\otimes_{\cO_X}k(x)\simeq\mathrm{Ker}(\delta_x)$ and $\mathrm{Ker}(\delta_X)$ is finite projective of rank $|J|$.\par
	The statement for general base changes, which we will not essentially need, follows form \cite[Lem. 4.1.5, Thm. 4.4.3 (2)]{kedlaya2014cohomology} and the locally-freeness of those sheaves over the base $X$. 
\end{proof}
The image of $\mathcal{S}_{n,(i_0,\mathbf{k}_J)}^{\square}(\overline{r})$ in $U_{\mathrm{tri}}(\overline{r})$ consists of $x=(r,\underline{\delta})$ such that $\wt_{\tau}(\delta_{i_0}\delta_{i_0+1}^{-1})=k_{\tau}$ for $\tau\in J$, $\delta_{i_0}\delta_{i_0+1}^{-1}\in\cT_0$ and the extension (the condition will be independent of the trivialization of $\cR_{k(x),K}(\delta_{i})$)
\[0\rightarrow\cR_{k(x),K}(\delta_{i_0})\rightarrow\Fil_{i_0+1}D_{\mathrm{rig}}(r)/\Fil_{i_0-1}D_{\mathrm{rig}}(r)\rightarrow\cR_{k(x),K}(\delta_{i_0+1})\rightarrow 0\]
corresponds to an element in $H^1_{\varphi,\gamma_K}(\cR_{k(x),K}(\delta_{i_0}\delta_{i_0+1}^{-1}))$ which lies in the kernel of 
\[H^1_{\varphi,\gamma_K}(\cR_{k(x),K}(\delta_{i_0}\delta_{i_0+1}^{-1}))\rightarrow H^1_{\varphi,\gamma_K}(t^{-\mathbf{k}_J}\cR_{k(x),K}(\delta_{i_0}\delta_{i_0+1}^{-1})).\]
The following lemma will be important.
\begin{lemma}\label{Lem:smooth}
	The morphism $\kappa:\mathcal{S}^{\square}_{n,(i_0,\mathbf{k}_J)}(\overline{r})\rightarrow \cT_{(i_0,\mathbf{k}_J)}^n$ is smooth.
\end{lemma}
\begin{proof}
	The diagram 
	\begin{center}
		\begin{tikzpicture}[scale=1.3]
			\node (A) at (0,1) {$\mathcal{S}^{\square}_{n,(i_0,\mathbf{k}_J)}(\overline{r})$};
			\node (B) at (2,1) {$\mathcal{S}_{n,(i_0,\mathbf{k}_J)}$};
			\node (C) at (0,0) {$\mathcal{S}^{\square}_{n}(\overline{r})$};
			\node (D) at (2,0) {$\mathcal{S}_n$};
			\path[->,font=\scriptsize,>=angle 90]
			(A) edge node[above]{} (B)
			(B) edge node[above]{} (D)
			(A) edge node[above]{} (C)
			(C) edge node[above]{} (D)
			;
			\end{tikzpicture}
	 \end{center}
	is Cartesian. Hence the map $\mathcal{S}^{\square}_{n,(i_0,\mathbf{k}_J)}(\overline{r})\rightarrow \mathcal{S}_{n,(i_0,\mathbf{k}_J)}$ is smooth.
	The map $\mathcal{S}^{\square}_{n,(i_0,\mathbf{k}_J)}(\overline{r})\rightarrow \cT_{(i_0,\mathbf{k}_J)}^n$ factors through $\mathcal{S}_{n,(i_0,\mathbf{k}_J)}\rightarrow \cT_{(i_0,\mathbf{k}_J)}^n$. Therefore, we only need to prove that $\mathcal{S}_{n,(i_0,\mathbf{k}_J)}\rightarrow \cT_{(i_0,\mathbf{k}_J)}^n$ is smooth. In \S\ref{subsec:trianguline}, we have maps $\mathcal{S}_i\rightarrow \mathcal{S}_{i-1}\times_L \mathcal{T}_L$. We can define $\mathcal{S}_{i_0+1,(i_0,\mathbf{k})}$ replacing $n$ by $i_0+1$. We have $\mathcal{T}_{(i_0,\mathbf{k}_J)}^n=\mathcal{T}^{n}_{\mathrm{reg}}\times_{\mathcal{T}^{i_0+1}_L}\mathcal{T}_{(i_0,\mathbf{k}_J)}^{i_0+1}$. The section $s_{\mathcal{S}_n\times_{\cT_{L}^{n}}\cT_{(i_0,\mathbf{k}_J)}^{n}}$ is the pullback of the section $s_{\mathcal{S}_{i_0+1}\times_{\cT_{L}^{i_0+1}}\cT_{(i_0,\mathbf{k}_J)}^{i_0+1}}$ via $\mathcal{S}_n\times_{\cT_{L}^{n}}\cT_{(i_0,\mathbf{k}_J)}^{n}\rightarrow \mathcal{S}_{i_0+1}\times_{\cT_{L}^{i_0+1}}\cT_{(i_0,\mathbf{k}_J)}^{i_0+1}$ since the definition of $s_X$ only involves $\Fil_{i_0+1}D_X$ and $\delta_{i_0},\delta_{i_0+1}$. Thus the diagram
	\begin{center}
		\begin{tikzpicture}[scale=1.3]
			\node (A) at (0,1) {$\mathcal{S}_{n,(i_0,\mathbf{k}_J)}$};
			\node (B) at (2,1) {$\mathcal{S}_{i_0+1,(i_0,\mathbf{k}_J)}$};
			\node (C) at (0,0) {$\mathcal{S}_n$};
			\node (D) at (2,0) {$\mathcal{S}_{i_0+1}$};
			\path[->,font=\scriptsize,>=angle 90]
			(A) edge node[above]{} (B)
			(B) edge node[above]{} (D)
			(A) edge node[above]{} (C)
			(C) edge node[above]{} (D)
			;
			\end{tikzpicture}
	 \end{center}
	 is Cartesian. As each $\mathcal{S}_{i}\rightarrow \mathcal{S}_{i-1}\times_L\mathcal{T}_L$ is smooth (as a geometric vector bundle over a Zariski open subset of the image), we see so is, by base change, $\mathcal{S}_{i,(i_0,\mathbf{k}_J)}\rightarrow \mathcal{S}_{i-1,(i_0,\mathbf{k}_J)}\times_L\mathcal{T}_L\rightarrow \mathcal{T}_{(i_0,\mathbf{k}_J)}^i\subset \mathcal{T}_{(i_0,\mathbf{k}_J)}^{i-1}\times_L\mathcal{T}_L$ for $i\geq i_0+2$ if the result is true for $i_0+1$. Thus, we reduce to the case when $n=i_0+1$. \par
	 We consider the map 
	 \[\mathcal{S}_{i_0+1,(i_0,\mathbf{k}_J)}\rightarrow \mathcal{S}_{i_0+1}\times_{\cT_{L}^{i_0+1}}\cT_{(i_0,\mathbf{k}_J)}^{i_0+1}\rightarrow (\mathcal{S}_{i_0}\times_L\mathcal{T}_L)\times_{\cT_{L}^{i_0+1}}\cT_{(i_0,\mathbf{k}_J)}^{i_0+1}.\]
	 Since the map $\mathcal{S}_{i_0}\rightarrow \mathcal{T}_L^{i_0}$ is smooth, so is $(\mathcal{S}_{i_0}\times_L\mathcal{T}_L)\times_{\cT_{L}^{i_0+1}}\cT_{(i_0,\mathbf{k}_J)}^{i_0+1}\rightarrow \cT_{(i_0,\mathbf{k}_J)}^{i_0+1}$. Write $V$ for $(\mathcal{S}_{i_0}\times_L\mathcal{T}_L)\times_{\cT_{L}^{i_0+1}}\cT_{(i_0,\mathbf{k}_J)}^{i_0+1}$. We only need to prove that $\mathcal{S}_{i_0+1,(i_0,\mathbf{k}_J)}$ is a geometric vector bundle over $V$ which will imply all we need.\par
	 Recall that $\mathcal{S}_{i_0+1}\times_{\cT_{L}^{i_0+1}}\cT_{(i_0,\mathbf{k}_J)}^{i_0+1}\simeq \mathrm{Spec}^{\mathrm{an}}(\mathrm{Sym}^{\bullet}(H^1_{\varphi,\gamma_K}(\Fil_{i_0}D_V(\delta_{i_0+1}^{-1}))^{\vee}))$ where $\Fil_{i_0}D_V$ is the universal one pulled back from $\mathcal{S}_{i_0}$ and $\delta_{i_0+1}$ is the character pulled back from $\mathcal{T}_{(i,\mathbf{k}_J)}^{i_0+1}$. Consider the kernel of the following composite of morphisms of coherent sheaves on $V$
	 \begin{equation}\label{equa:coherentsheaves}
		H^1_{\varphi,\gamma_K}(\Fil_{i_0}D_V(\delta_{i_0+1}^{-1}))\rightarrow H^1_{\varphi,\gamma_K}(\cR_{V,K}(\delta_{i_0}\delta_{i_0+1}^{-1}))\rightarrow H^1_{\varphi,\gamma_K}(t^{-\mathbf{k}_J}\cR_{V,K}(\delta_{i_0}\delta_{i_0+1}^{-1}))
	 \end{equation}
	 where $ \cR_{V,K}(\delta_{i_0})=\Fil_{i_0}D_V/\Fil_{i_0-1}D_V$. We denote the kernel (resp. cokernel) of the above composite morphism by $\mathrm{Ker}(\Fil_{i_0}D_V(\delta_{i_0+1}^{-1}))$ (resp. $\mathrm{Coker}(\Fil_{i_0}D_V(\delta_{i_0+1}^{-1}))$). \par
	 We claim that $\mathrm{Ker}(\Fil_{i_0}D_V(\delta_{i_0+1}^{-1}))$ is locally free of rank $(i_0-1)|\Sigma|+|J|$ and 
	 \[\mathrm{Ker}(\Fil_{i_0}D_V(\delta_{i_0+1}^{-1}))\otimes_{\cO_V}k(x)\simeq \mathrm{Ker}(\Fil_{i_0}D_x(\delta_{x,i_0+1}^{-1}))\]
	  for any $x\in V$. If $i_0=1$, this follows from Lemma \ref{lem:locallyfree}. Now we assume $i_0-1\geq 1$. The sheaves $H^1_{\varphi,\gamma_K}(\Fil_{i_0}D_V(\delta_{i_0+1}^{-1}))$, $H^1_{\varphi,\gamma_K}(\cR_{V,K}(\delta_{i_0}\delta_{i_0+1}^{-1}))$ and $H^1_{\varphi,\gamma_K}(t^{-\mathbf{k}_J}\cR_{V,K}(\delta_{i_0}\delta_{i_0+1}^{-1}))$ are locally free of ranks $i_0|\Sigma|$, $|\Sigma|$ and $|\Sigma|$ respectively and commute with base change. The morphism
	 \[H^1_{\varphi,\gamma_K}(\Fil_{i_0}D_V(\delta_{i_0+1}^{-1}))\rightarrow H^1_{\varphi,\gamma_K}(\cR_{V,K}(\delta_{i_0}\delta_{i_0+1}^{-1}))\]
	 is surjective since for any $x\in V$, $H^2_{\varphi,\gamma_K}(\Fil_{i_0-1}D_x(\delta_{x,i_0+1}^{-1}))=0$ (see \cite[Prop. 2.3]{hellmann2016density}). Hence 
	 \[\mathrm{Coker}(\Fil_{i_0}D_V(\delta_{i_0+1}^{-1}))=\mathrm{Coker}(H^1_{\varphi,\gamma_K}(\cR_{V,K}(\delta_{i_0}\delta_{i_0+1}^{-1}))\rightarrow H^1_{\varphi,\gamma_K}(t^{-\mathbf{k}_J}\cR_{V,K}(\delta_{i_0}\delta_{i_0+1}^{-1}))).\]
	 By Lemma \ref{lem:locallyfree}, we get $\mathrm{Coker}(\Fil_{i_0}D_V(\delta_{i_0+1}^{-1}))$ is locally free of rank $|J|$ and for any point $x\in V$, $\mathrm{Coker}(\Fil_{i_0}D_V(\delta_{i_0+1}^{-1}))\otimes_{\cO_V}k(x)\simeq \mathrm{Coker}(\Fil_{i_0}D_x(\delta_{x,i_0+1}^{-1}))$. Repeat the last step of the proof of Lemma \ref{lem:locallyfree}, we get the desired claim. \par
	 The injection $\mathrm{Ker}(\Fil_{i_0}D_V(\delta_{i_0+1}^{-1}))\hookrightarrow H^1_{\varphi,\gamma_K}(\Fil_{i_0}D_V(\delta_{i_0+1}^{-1}))$ of projective coherent sheaves induces a surjection 
	 \[H^1_{\varphi,\gamma_K}(\Fil_{i_0}D_V(\delta_{i_0+1}^{-1}))^{\vee}\twoheadrightarrow \mathrm{Ker}(\Fil_{i_0}D_V(\delta_{i_0+1}^{-1}))^{\vee}\] 
	 which by \cite[Thm. 2.2.5]{conrad2006relative} induces a closed embedding
	 \[\mathrm{Spec}^{\mathrm{an}}(\mathrm{Sym}^{\bullet}(\mathrm{Ker}(\Fil_{i_0}D_V(\delta_{i_0+1}^{-1}))^{\vee}))\hookrightarrow \mathcal{S}_{i_0+1}\times_{\cT_{L}^{i_0+1}}\cT_{(i_0,\mathbf{k}_J)}^{i_0+1}.\]
	 The left-hand side is a geometric vector bundle over $V$ by the previous discussion, and we remain to prove that $\mathrm{Spec}^{\mathrm{an}}(\mathrm{Sym}^{\bullet}(\mathrm{Ker}(\Fil_{i_0}D_V(\delta_{i_0+1}^{-1}))^{\vee}))$ coincides with $\mathcal{S}_{i_0+1,(i_0,\mathbf{k}_J)}$. The statement is local and trivial. We write a proof below.\par
	 We may take an affinoid open $W=\Sp(A)\subset V$ and assume that the sheaves in (\ref{equa:coherentsheaves}) are free over $W$. Then since all the modules are projective, we may take a basis $e_1,\cdots,e_{i_0|\Sigma|}$ of $H^1_{\varphi,\gamma_K}(\Fil_{i_0}D_W(\delta_{i_0+1}^{-1}))$ and assume that the surjection 
	 \[H^1_{\varphi,\gamma_K}(\Fil_{i_0}D_W(\delta_{i_0+1}^{-1}))\twoheadrightarrow H^1_{\varphi,\gamma_K}(\cR_{W,K}(\delta_{i_0}\delta_{i_0+1}^{-1}))\] 
	 corresponds to projection to the subspace $\langle e_1,\cdots,e_{|\Sigma|}\rangle$ (equivalently choose a split of the surjection). We assume that $e_1',\cdots,e_{|\Sigma|}'$ is a basis of $H^1_{\varphi,\gamma_K}(t^{-\mathbf{k}_J}\cR_{W,K}(\delta_{i_0}\delta_{i_0+1}^{-1}))$. As the cokernel and the kernel of the map $H^1_{\varphi,\gamma_K}(\cR_{W,K}(\delta_{i_0}\delta_{i_0+1}^{-1}))\rightarrow H^1_{\varphi,\gamma_K}(t^{-\mathbf{k}_J}\cR_{W,K}(\delta_{i_0}\delta_{i_0+1}^{-1}))$ are locally free, we may, after possibly shrinking $W$, assume that the morphism is given by sending $e_{|J|+1},\cdots,e_{|\Sigma|}$ to $e_{|J|+1}',\cdots, e_{|\Sigma|}'$ and sending $e_{1},\cdots, e_{|J|}$ to $0$. Let $e_1^{\vee},\cdots, e_{i_0|\Sigma|}^{\vee}$ be the dual basis. Then 
	 \[\mathrm{Spec}^{\mathrm{an}}(\mathrm{Sym}^{\bullet}(H^1_{\varphi,\gamma_K}(\Fil_{i_0}D_V(\delta_{i_0+1}^{-1}))^{\vee}))\text{ resp. }\mathrm{Spec}^{\mathrm{an}}(\mathrm{Sym}^{\bullet}(\mathrm{Ker}(\Fil_{i_0}D_V(\delta_{i_0+1}^{-1}))^{\vee}))\] are covered by \[W_N:=\Sp(A\langle p^Ne_1^{\vee},\cdots, p^Ne_{i_0|\Sigma|}^{\vee}\rangle)\text{ resp. }\Sp(A\langle p^Ne_1^{\vee},\cdots,p^Ne_{|J|}^{\vee},p^Ne_{|\Sigma|+1}^{\vee},\cdots, p^Ne_{i_0|\Sigma|}^{\vee}\rangle)\] 
	 where $N\in \N$. The tautological section $s_{W_N}$ of the sheaf 
	 \[H^1_{\varphi,\gamma_K}(\cR_{W_N,K}(\delta_{i_0}\delta_{i_0+1}^{-1}))=\cO_{W_N}e_1\oplus\cdots\oplus \cO_{W_N}e_{|\Sigma|}\] 
	 is given by $e_1^{\vee}e_1+\cdots+e_{|\Sigma|}^{\vee}e_{|\Sigma|}$. Thus the image of $s_{W_N}$ in 
	 \[H^1_{\varphi,\gamma_K}(t^{-\mathbf{k}_J}\cR_{W_N,K}(\delta_{i_0}\delta_{i_0+1}^{-1}))=\cO_{W_N}e_1'\oplus\cdots\oplus \cO_{W_N}e_{|\Sigma|}'\] 
	 is given by $e_{|J|+1}^{\vee}e_{|J|+1}'+\cdots+e_{|\Sigma|}^{\vee}e_{|\Sigma|}'$. Hence the vanishing locus is cut out by $e_{|J|+1}^{\vee}=\cdots=e_{|\Sigma|}^{\vee}=0$ and coincides with $\Sp(A\langle p^Ne_1^{\vee},\cdots,p^Ne_{|J|}^{\vee},p^Ne_{|\Sigma|+1}^{\vee},\cdots, p^Ne_{i_0|\Sigma|}^{\vee}\rangle)$.
\end{proof}
\subsection{Nearby critical crystalline points}
Now we fix $x=(r,\underline{\delta})\in U_{\mathrm{tri}}(\overline{r})(L)\subset (\fX_{\overline{r}}\times \cT_L^n)(L)$. We assume that $r$ is crystalline and $\underline{\delta}=z^{\lambda}\mathrm{unr}(\underline{\varphi}):=(z^{\lambda_i}\mathrm{unr}(\varphi_i))_{i=1,\cdots,n}$ for some $\lambda=(\lambda_{\tau,i})_{\tau\in\Sigma,i=1,\cdots,n}\in(\Z^n)^{\Sigma}$ and $\underline{\varphi}\in (L^{\times})^n$. Assume furthermore that $\varphi_i\varphi_j^{-1}\notin \{1, q\}$ for all $i\neq j$ where $q$ is the cardinal of the residue field of $\cO_K$. This means that $r$ is generic in the sense of \cite[\S4.1]{wu2021local} or the beginning of \S\ref{sec:companionpointsdesciption}, and $\underline{\delta}\in \cT_0^n$. \par
We continue to fix $i_0\in \{1,\cdots,n-1\}$. Let $J:=\{\tau\in \Sigma\mid \lambda_{\tau,i_0}\geq \lambda_{\tau,i_0+1}+1\}$ and let $\mathbf{k}_J=(k_{\tau})_{\tau\in J}:=(\lambda_{\tau,i_0}- \lambda_{\tau,i_0+1})_{\tau\in J}$. We have a filtration $\Fil_{\bullet}D_{\mathrm{rig}}(r)$ such that $\Fil_{i}D_{\mathrm{rig}}(r)/\Fil_{i-1}D_{\mathrm{rig}}(r)\simeq \cR_{L,K}(\delta_i)$. Since $r$ is de Rham, so are $D_{\mathrm{rig}}(r)$ and the subquotient $\Fil_{i_0+1}D_{\mathrm{rig}}(r)/\Fil_{i_0-1}D_{\mathrm{rig}}(r)$. The extension 
\[0\rightarrow \cR_{L,K}(\delta_{i_0})\rightarrow \Fil_{i_0+1}D_{\mathrm{rig}}(r)/\Fil_{i_0-1}D_{\mathrm{rig}}(r)\rightarrow \cR_{L,K}(\delta_{i_0+1})\rightarrow 0\]
defines an element in $H^1_{\varphi,\gamma_K}(\cR_{L,K}(\delta_{i_0}\delta_{i_0+1}^{-1}))$ (up to $L^{\times}$) which lies in the kernel of 
\[H^{1}_{\varphi,\gamma_K}(\cR_{L,K}(\delta_{i_0}\delta_{i_0+1}^{-1}))\rightarrow H^{1}_{\varphi,\gamma_K}(t^{-\mathbf{k}_J}\cR_{L,K}(\delta_{i_0}\delta_{i_0+1}^{-1}))\]
by Lemma \ref{lem:cohomologyderhamphigamma}, Proposition \ref{prop:kernelandderham}, \cite[Lem. 3.3.7, Lemma 3.4.2]{breuil2019local}, the isomorphism \[H^1(\cG_K, L\otimes_{\Q_p}\BdR)\simeq \mathrm{Ext}^1_{\mathrm{Rep}_{L\otimes_{\Q_p}\BdR}(\cG_K)}(L\otimes_{\Q_p}\BdR,L\otimes_{\Q_p}\BdR)\]
and the following lemma.
\begin{lemma}
	Let $W$ be an $L\otimes_{\Q_p}\BdR$-representation of $\cG_K$ which is an extension $0\rightarrow L\otimes_{\Q_p}\BdR\rightarrow W\rightarrow L\otimes_{\Q_p}\BdR\rightarrow 0$ as representations of $\cG_K$. Then $W$ is trivial (i.e. $W \simeq (L\otimes_{\Q_p}\BdR)^2$) if and only if the extension splits.
\end{lemma}
\begin{proof}
	If the extension splits, then $W$ is trivial. Conversely, if $W$ is trivial, then $\dim_LW^{\cG_K}=2[K:\Q_p]$ and we have an exact sequence of $L\otimes_{\Q_p}K$-modules
	$0\rightarrow L\otimes_{\Q_p}K\rightarrow W^{\cG_K}\rightarrow L\otimes_{\Q_p}K\rightarrow 0.$
	The extension splits and we may choose a section $L\otimes_{\Q_p}K\rightarrow W^{\cG_K}$ which induces a section $L\otimes_{\Q_p}\BdR\rightarrow W$ of $L\otimes_{\Q_p}\BdR$-representations.
\end{proof}
Thus $x$ lies in the image of $\mathcal{S}_{n,(i_0,\mathbf{k}_J)}^{\square}(\overline{r})$. Recall the following diagram.
\begin{center}
	\begin{tikzpicture}[scale=1.3]
		\node (A) at (2,1) {$\mathcal{S}^{\square}_{n,(i_0,\mathbf{k}_J)}(\overline{r})\subset \mathcal{S}^{\square}_{n}(\overline{r})$};
		\node (C) at (0,0) {$x\in U_{\mathrm{tri}}(\overline{r})$};
		\node (D) at (4,0) {$\cT_{(i_0,\mathbf{k}_J)}^n\subset \cT_L^n.$};
		\path[->,font=\scriptsize,>=angle 90]
		
		(A) edge node[above]{$\pi_{\overline{r}}$} (C)
		(A) edge node[above]{$\kappa$} (D)
		(C) edge node[above]{$\omega'$} (D)
		;
		\end{tikzpicture}
 \end{center}
 \begin{definition}
	Let $A$ and $B$ be two subsets of a rigid space $X$ over $L$. Then we say that $A$ \emph{quasi-accumulates} at $B$ if for every point $b\in B$ and every affinoid open neighbourhood $Y$ of $b$, $A\cap Y\neq \emptyset$ (compare with \cite[Def. 2.2]{breuil2017interpretation}).
\end{definition}
\begin{lemma}\label{lem:quasi-accumulateszariskiclosure}
	If $A$ and $B$ are two subsets of a rigid space $X$, then $A$ quasi-accumulates at $B$ if and only if for any $b\in B$ and any affinoid open neighbourhood $Y$ of $b$, $b$ lies in the Zariski closure of $Y\cap A$ in $Y$. In particular, if $A$ quasi-accumulates at $B$, then $B$ is contained in the Zariski closure of $A$ in $X$.
\end{lemma}
\begin{proof}
	We prove by contradiction. Assume that $A$ quasi-accumulates at $B$ and there exists an affinoid neighbourhood $Y$ of $b\in B$ such that $b$ is not in the Zariski closure $\overline{Y\cap A}$ in $Y$. Since Zariski open subsets in an affinoid are admissible open (\cite[Cor. 5.1.9]{bosch2014lectures}), there exists an affinoid neighbourhood $Y'\subset Y\setminus \overline{Y\cap A}$ of $b$. Then $Y'\cap A=\emptyset$, this contradicts the assumption.
\end{proof}
\begin{lemma}\label{lem: quasi-accumulates}
	Let $Y\hookrightarrow X$ be a closed immersion of rigid analytic spaces over $L$. Let $Z$ be a subset of $Y$ and $y\in Y$ be a point. Then $Z$ quasi-accumulates at $y$ in $X$ if and only if $Z$ quasi-accumulates at $y$ in $Y$.
\end{lemma}
\begin{proof}
	The problem is local and we may assume $X=\Sp(A), Y=\Sp(B)$ and $B=A/I$ for an ideal $I$. Assume that $Z$ quasi-accumulates at $y$ in $X$. We only need to prove that for any affinoid neighbourhood $Y'$ of $y$ in $Y$, there exists an affinoid neighbourhood $Y''\subset Y'$ such that $Y''$ has the form $X'\cap Y$ for some affinoid neighbourhood $X'$ of $y$ in $X$. As affinoid subdomains are open in the canonical topology (\cite[Prop. 3.3.19]{bosch2014lectures}) and Weierstrass domains form a basis of the canonical topology (\cite[Lem. 3.3.8]{bosch2014lectures}), we may assume that $Y''$ has the form $\{x\in Y\mid |f_i(x)|\leq 1\}$ for $f_1,\cdots,f_m\in A/I$. We may choose lifts $ \widetilde{f}_1,\cdots,\widetilde{f}_m$ for $f_1,\cdots,f_m$ in $A$. Then $\{x\in X\mid |\widetilde{f}_i(x)|\leq 1\}\cap Y=\{x\in Y\mid |f_i(x)|\leq 1\}$. 
\end{proof}
\begin{lemma}\label{lem:characters}
	Let $C$ be a positive integer. Then the set of crystalline characters $\underline{\delta}\in\cT_L^n$ such that, if we write $\lambda=\wt(\underline{\delta})$, $\lambda_{\tau,i}- \lambda_{\tau,i+1}>C$ if $i\neq i_0$, $\lambda_{i}-\lambda_{i_0}>C,\lambda_{i}-\lambda_{i_0+1}>C$ if $i<i_0$, $\lambda_{i}-\lambda_{i_0}<-C,\lambda_{i}-\lambda_{i_0+1}<-C$ if $i>i_0+1$, $\lambda_{\tau,i_0}=\lambda_{\tau,i_0+1}$ if $\tau\in J$ and $\lambda_{\tau,i_0+1}-\lambda_{\tau,i_0}>C$ quasi-accumulates at the trivial character in $\cT_L^n$
\end{lemma}
\begin{proof}
	Let $q=p^{[K_0:\Q_p]}$, where $K_0$ is the maximal unramified subfield of $K$, and $d=|\Sigma|$. Take a uniformizer $\varpi_K$ of $K$. We prove that for any character $\delta\in\cT_L$ such that $\delta(\varpi_K)=1$, the set $\{ \delta^{p^N(q-1)},N\in\N\}$ quasi-accumulates at the trivial character. For some $m$ large, we have $\cO_K^{\times}=\Z_p^d\times \mu(\cO_K)\times \Z/(q-1)$ where $\Z_p^d/\exp(\varpi_K^m\cO_K)$ is finite and $\mu(\cO_K)$ denotes the $p$-power roots of unity in $\cO_K$ (see \cite[Prop. II.5.7]{neukirch2013algebraic}). We only need to consider $\Z_p^d=\Z_pe_1\oplus\cdots\oplus\Z_pe_n$ since the characters $\delta^{p^{N}(q-1)}$ are trivial on the torsion subgroups of $\cO_K^{\times}$ and $\varpi_K^{\Z}$ when $N$ is large. The space $\widehat{\Z_p^d}=\mathbb{U}^d$ which parametrizes characters of $\Z_p^d$ is the open polydisk in $d$ variables $T_1,\cdots,T_d$ by sending a character $\delta$ to $(\delta(e_1)-1,\cdots,\delta(e_d)-1)$. Then $\delta^{p^N}$ is sent to $(\delta(e_1)^{p^N}-1,\cdots,\delta(e_d)^{p^N}-1)$. For any $x\in C$ such that $|x-1|_p<1$, where $|-|_p$ denotes the standard valuation, $\lim_{N\to\infty }|x^{p^N}-1|_p=\lim_{N\to\infty}|\sum_{1\leq i\leq p^N}\binom{p^N}{i}(x-1)^i|_p=0$. Hence for any $\epsilon>0$ and $N$ large enough, we have $(\delta(e_1)^{p^N}-1,\cdots,\delta(e_d)^{p^N}-1)\in\overline{\mathbb{B}}(0,\epsilon)^d:=\{x\in \mathbb{U}^d\mid |T_1(x)|_p\leq\epsilon,\cdots, |T_d(x)|_p\leq\epsilon\}$. Any affinoid neighbourhood of $0$ in $\overline{\mathbb{B}}(0,\frac{1}{p})^d$ contains a Weierstrass subdomain of the form $\{x\in \overline{\mathbb{B}}(0,\frac{1}{p})^d \mid |f_1(x)|_p\leq 1,\cdots |f_m(x)|_p\leq 1\}$ for some $f_1,\cdots,f_m\in L\langle p^{-1}T_1,\cdots,p^{-1}T_d\rangle$ by \cite[Lem. 3.3.8, Prop. 3.3.19]{bosch2014lectures}. Since $|f_i(0)|_p\leq 1$, there exists $\epsilon>0$ satisfying that for all $(x_1,\cdots,x_d)\in C^d$ such that $|p^{-1}x_i|_p<\epsilon$ for all $i=1,\cdots,d$, $|f_i(x_1,\cdots,x_n)|_p\leq 1$. Hence $\overline{\mathbb{B}}(0,\epsilon)^d \subset  \{x\in \overline{\mathbb{B}}(0,\frac{1}{p})^d \mid |f_1(x)|_p\leq 1,\cdots |f_m(x)|_p\leq 1\}$. Therefore, we have $\{\delta^{p^N(q-1)},N\in\N\}$ quasi-accumulates at the trivial character.\par
	We write $x_{\tau}$ for the character that sends $x\in \cO_K^{\times}$ to $\tau(x)$ and $\varpi_{K}$ to $1$. For $i\neq i_0,i_0+1$, let $\delta_i=\prod_{\tau\in\Sigma}x_{\tau}^{-i}$. Let $\delta_{i_0}=\prod_{\tau\in J}x_{\tau}^{-i_0}\prod_{\tau\notin J}x_{\tau}^{-i_0-1}$ and $\delta_{i_0+1}=\prod_{\tau\in J}x_{\tau}^{-i_0}\prod_{\tau\notin J}x_{\tau}^{-i_0}$. Let $\underline{\delta}=(\delta_1,\cdots,\delta_n)$. Then $\underline{\delta}$, as well as its powers, is crystalline. The set $\{\underline{\delta}^{p^N(q-1)}\mid N\in\N \}$ quasi-accumulates at the trivial character by a similar proof as above and $\underline{\delta}^{p^N(q-1)}$ satisfies the requirements for the weights if $N$ is large.
\end{proof}
Finally we can prove the main local results. Write $z^{\mathbf{k}_J}$ for the character $K^{\times}\rightarrow L^{\times}: z\mapsto \prod_{\tau\in J}\tau(z)^{k_{\tau}}$.
\begin{proposition}\label{prop:key}
	Let $X$ be an affinoid open neighbourhood of $x$ in $U_{\mathrm{tri}}(\overline{r})$. 
	\begin{enumerate}
		\item There exists a subset $Z\subset X$ that quasi-accumulates at $x$ and such that for every $z=(r_z,\underline{\delta}_z)\in Z$,
		\begin{enumerate}
			\item $z$ lies in the image of $\mathcal{S}^{\square}_{n,(i_0,\mathbf{k}_J)}(\overline{r})\subset \mathcal{S}^{\square}_{n}$,
			\item $\underline{\delta}_z\in\cT_0^n$ is crystalline, and
			\item if we write $\lambda_z$ for $\wt(\underline{\delta}_z)$, then for every $\tau\in\Sigma$, $\lambda_{z,\tau,i}> \lambda_{z,\tau,i+1}$ if $i\neq i_0$, $\lambda_{z,\tau,i}> \lambda_{z,\tau,i_0},\lambda_{z,\tau,i_0+1}$ if $i<i_0$, $\lambda_{z,\tau,i}<\lambda_{z,\tau,i_0},\lambda_{z,\tau,i_0+1}$ if $i>i_0+1$, and $\lambda_{z,\tau,i_0}<\lambda_{z,\tau,i_0+1}$ if $\tau\notin J$.
		\end{enumerate}
		\item Every point in $Z$ is generic crystalline and regular (i.e. $\lambda_{z,\tau,i}\neq \lambda_{z,\tau,j}$ for all $i\neq j,\tau\in\Sigma$).
		\item Let $\zeta$ be the automorphism of $\cT_L^n$ sending $\underline{\delta}'=(\delta_1',\cdots, \delta_n')$ to 
		\[(\delta_1',\cdots,\delta_{i_0-1}', \delta_{i_0+1}'z^{\mathbf{k}_J},\delta_{i_0}'z^{-\mathbf{k}_J},\delta'_{i_0+2},\cdots,\delta'_{n}).\] 
		Use also the notation $\zeta$ to denote the automorphism of $\fX_{\overline{r}}\times \cT_L^n:(r,\underline{\delta}')\mapsto (r,\zeta(\underline{\delta}'))$. Then $\zeta(Z)$ is a subset of $X_{\mathrm{tri}}(\overline{r})$ and quasi-accumulates at $\zeta(x)$ in $X_{\mathrm{tri}}(\overline{r})$.
	\end{enumerate}
\end{proposition}
\begin{proof}
	(1) By the definition of quasi-accumulation, we only need to verify that there exists one point $z\in X$ for an arbitrary affinoid open neighbourhood $X\subset U_{\mathrm{tri}}(\overline{r})$ satisfying the condition (a), (b) and (c). Let $\pi_{\overline{r}}^{-1}(X)$ be the preimage of $X$ in $\mathcal{S}^{\square}_{n,(i_0,\mathbf{k}_J)}(\overline{r})$. Then $\pi_{\overline{r}}^{-1}(X)$ is admissible open. We only need to prove that there exists a point $z'\in \pi_{\overline{r}}^{-1}(X)$ such that $\kappa(z')=\omega'(\pi_{\overline{r}}(z'))\in \cT_{(i_0,\mathbf{k}_J)}$ satisfies the conditions in (b) and (c). As $\kappa:\mathcal{S}^{\square}_{n,(i_0,\mathbf{k}_J)}(\overline{r})\rightarrow \cT_{(i_0,\mathbf{k}_J)}^n$ is smooth by Lemma \ref{Lem:smooth}, the image $\kappa(\pi_{\overline{r}}^{-1}(X))$, which contains $\underline{\delta}$, contains an admissible open subset of $\cT_{(i_0,\mathbf{k}_J)}^n$ that contains $\underline{\delta}$ by \cite[Cor. 9.4.2]{bosch2014lectures}. Then the result follows from that the set of points $\underline{\delta}'\in \cT_{(i_0,\mathbf{k}_J)}^n$ that satisfy (b) and (c) quasi-accumulates at $\underline{\delta}$ by Lemma \ref{lem:characters} (since $\underline{\delta}\in \cT_0^n$ and $\cT_L^n\rightarrow\cT_L^n: \underline{\delta}'\mapsto \underline{\delta}\underline{\delta}'$ is an isomorphism).\par
	(2) Assume $z=(r_z,\underline{\delta}_z)\in Z$ as in (1). By (c), the $\tau$-weights of $\underline{\delta}_z$ are pairwise different, thus the Sen weights of $r_z$ are regular. Since $(r_z,\underline{\delta}_z)$ lies in the image of $\mathcal{S}^{\square}_{n,(i_0,\mathbf{k}_J)}(\overline{r})$, the extension
	\[0\rightarrow \cR_{k(z),K}(\delta_{z,i_0})\rightarrow \Fil_{i_0+1}D_{\mathrm{rig}}(r_z)/\Fil_{i_0-1}D_{\mathrm{rig}}(r_z)\rightarrow \cR_{k(z),K}(\delta_{z,i_0+1})\rightarrow 0\]
	corresponds to an element (up to $L^{\times}$) in the kernel of 
	\[H^{1}_{\varphi,\gamma_K}(\cR_{k(z),K}(\delta_{z,i_0}\delta_{z,i_0+1}^{-1}))\rightarrow H^{1}_{\varphi,\gamma_K}(t^{-\mathbf{k}_J}\cR_{k(z),K}(\delta_{z,i_0}\delta_{z,i_0+1}^{-1}))\]
	and in particular, in the kernel of 
	\[H^{1}_{\varphi,\gamma_K}(\cR_{k(z),K}(\delta_{z,i_0}\delta_{z,i_0+1}^{-1}))\rightarrow H^{1}_{\varphi,\gamma_K}(\cR_{k(z),K}(\delta_{z,i_0}\delta_{z,i_0+1}^{-1})[\frac{1}{t}]).\]
	Since $\delta_{z,i_0},\delta_{z,i_0+1}$ are both locally algebraic, we get that $\Fil_{i_0+1}D_{\mathrm{rig}}(r_z)/\Fil_{i_0-1}D_{\mathrm{rig}}(r_z)[\frac{1}{t}]$ is a direct sum of de Rham $(\varphi,\Gamma_K)$-modules over $\cR_{k(z),K}[\frac{1}{t}]$ by \cite[Lem. 3.3.7]{breuil2019local}. Hence the $(\varphi,\Gamma_K)$-module $\Fil_{i_0+1}D_{\mathrm{rig}}(r_z)/\Fil_{i_0-1}D_{\mathrm{rig}}(r_z)$ over $\cR_{k(z),K}$ is de Rham. By Proposition \ref{prop:de Rham} and the condition of weights in (c), $r_z$ is de Rham. By (b) and Lemma \ref{lem:crystalline}, $r_z$ is generic crystalline.\par
	(3) Let $z=(\rho_z,\underline{\delta}_z)\in Z\subset U_{\mathrm{tri}}(\overline{r})$. Then $\underline{\delta}_z=z^{\lambda_z}\mathrm{unr}(\underline{\varphi}_z)$ for a refinement $\underline{\varphi}_z=(\varphi_{z,1},\cdots,\varphi_{z,n})$ where $\lambda_z$ is as in (c) (we abuse the notation $z$ for a point and the character). Let $\underline{\varphi}_z'$ be the refinement such that $\varphi_{z,i}'=\varphi_{z,i}$ if $i\neq i_0,i_0+1$ and $\varphi_{z,i_0}'=\varphi_{z,i_0+1},\varphi_{z,i_0+1}'=\varphi_{z,i_0}$. Let $\lambda^{\mathrm{dom}}_z$ be the weight such that $\lambda_{z,\tau,i}^{\mathrm{dom}}=\lambda_{z,\tau,i}$ if $i\neq i_0,i_0+1$ or if $\tau\in J$ and let $\lambda_{z,\tau,i_0}^{\mathrm{dom}}=\lambda_{z,\tau,i_0+1}, \lambda_{z,\tau,i_0+1}^{\mathrm{dom}}=\lambda_{z,\tau,i_0}$ if $\tau\notin J$. Then $\lambda^{\mathrm{dom}}_z$ is dominant and differs from $ \lambda_z$ by permutations. It is easy to verify that $z^{\lambda_{z}^{\mathrm{dom}}}\mathrm{unr}(\underline{\varphi}'_z)=\zeta(z^{\lambda_z}\mathrm{unr}(\underline{\varphi}_z))$. By \cite[Thm. 4.2.3]{breuil2019local}, all the companion points of $z$ exist on $X_{\mathrm{tri}}(\overline{r})$. In particular the dominant point $(r_z,z^{\lambda_{z}^{\mathrm{dom}}}\mathrm{unr}(\underline{\varphi}'_z))$ corresponding to the refinement $\underline{\varphi}'_z$ is on $X_{\mathrm{tri}}(\overline{r})$. Let $z$ vary, we see $\zeta(Z)\subset X_{\mathrm{tri}}(\overline{r})$. Since $Z$ quasi-accumulates at $x$ in $U_{\mathrm{tri}}(\overline{r})\subset X_{\mathrm{tri}}(\overline{r})$, $Z$ quasi-accumulates at $x$ in $\fX_{\overline{r}}\times \cT_L^n$ by Lemma \ref{lem: quasi-accumulates}. Since $\zeta$ is an automorphism, $\zeta(Z)$ quasi-accumulates at $\zeta(x)$ in $\fX_{\overline{r}}\times \cT_L^n$. As $\zeta(Z)\subset X_{\mathrm{tri}}(\overline{r})$, we see $\zeta(Z)$ quasi-accumulates at $\zeta(x)\in X_{\mathrm{tri}}(\overline{r})$ by Lemma \ref{lem:quasi-accumulateszariskiclosure}.
\end{proof}
\section{Companion points on the eigenvariety}\label{sec:global}
We now prove the existence of all companion points for generic crystalline points on the eigenvariety. We recall the definition of the eigenvariety for definite unitary groups in \cite[\S5.1]{breuil2019local} or \cite[\S3.1]{breuil2017smoothness}.
\subsection{The eigenvariety}
Let $F$ be a quadratic imaginary extension of a totally real field $F^+$. Let $S_p$ be the set of places of $F^+$ that divide $p$. We assume that each $v\in S_p$ splits in $F$ and for every $v\in S_p$, we choose a place $\widetilde{v}$ of $F$ above $v$. Let $G$ be a definite unitary group of rank $n\geq 2$ over $F^+$ that is split over $F$ so that $G_p=\prod_{v\in S_p}G_v:=G(F^+\otimes_{\Q}\Q_p)\simeq \prod_{v\in S_p}\GL_n(F_{\widetilde{v}})$ (we fix an isomorphism $G\times_{F^+}F\simeq\GL_{n/F}$). Let $B_p=\prod_{v\in S_p}B_v$ be the subgroup of upper triangular matrices in $G_p$ and let $T_p=\prod_{v\in S_p}T_v\subset B_p$ be the diagonal torus. Let $U^p=\prod_{v\nmid p} U_v$ be a sufficiently small (see \cite[(3.9)]{breuil2017smoothness}) open compact subgroup of $G(\mathbf{A}_{F^+}^{p\infty})$. Write $\widehat{S}(U^p,L):=\{f:G(F^+)\setminus G(\mathbf{A}^{\infty}_{F^+})/U^p\rightarrow L,\text{ continuous}\}$, where $L/\Q_p$ is a large enough finite extension with the residue field $k_L$. Let $G_p$ acts by right translations on $\widehat{S}(U^p,L)$. Let $S\supset S_p$ be a finite set of places of $F^+$ that split in $F$ and contains all such split places $v$ such that $U_v$ is not maximal. The space $\widehat{S}(U^p,L)$ is also endowed with some usual action of (away from $S$) Hecke operators and one can talk about the $p$-adic representations of $\cG_F:=\Gal(\overline{F}/F)$ associated with Hecke eigenvalues that appear in $\widehat{S}(U^p,L)$. We fix a modular absolutely irreducible $\overline{\rho}:\cG_F\rightarrow \GL_n(k_L)$ and write $\widehat{S}(U^p,L)_{\overline{\rho}}\neq 0$ for the localization of $\widehat{S}(U^p,L)_{\overline{\rho}}$ at the non-Eisenstein maximal ideal of the Hecke algebra over $\cO_L$ associated with $\overline{\rho}$ (see \cite[\S2.4]{breuil2017interpretation} for details). We assume the following ``standard Taylor-Wiles hypothesis''.
\begin{assumption}\label{ass:taylorwiles}
\begin{enumerate}
	\item $p>2$;
	\item $F$ is an unramified extension of $F^+$;
	\item $G$ is quasi-split at all finite places of $F^{+}$;
	\item $U_v$ is hyperspecial at all places $v$ of $F^+$ that are inert in $F$;
	\item $F$ contains no non-trivial $\sqrt[p]{1}$ and the image of $\overline{\rho}\mid_{\Gal(\overline{F}/F(\sqrt[p]{1}))}$ is adequate, see \cite[Rem. 1.1]{breuil2019local}.
\end{enumerate}
\end{assumption}
Let $R_{\overline{\rho},S}$ be the deformation ring of polarized deformations of $\overline{\rho}$ that are unramified outside $S$. This is a Noetherian complete local ring over $\cO_L$ with residue field $k_L$. We have an action of $R_{\overline{\rho},S}$ over $\widehat{S}(U^p,L)_{\overline{\rho}}$ which factors through the Hecke actions and commutes with that of $G_p$ (for details, see also \cite[\S2.4]{breuil2017interpretation}). Let $\mathrm{Spf}(R_{\overline{\rho},S})^{\mathrm{rig}}$ denote the rigid generic fiber of the formal scheme $\mathrm{Spf}(R_{\overline{\rho},S})^{\mathrm{rig}}$ in the sense of Berthelot, cf. \cite[\S7]{de1995crystalline}. Let $\widehat{T}_{p}$ be the rigid space over $\Q_p$ parametrizing continuous characters of $T_{p}$ and we write $\widehat{T}_{p,L}$ for its base change. Denote by $\widehat{S}(U^p,L)_{\overline{\rho}}^{\mathrm{an}}$ the subspace of $\widehat{S}(U^p,L)_{\overline{\rho}}$ consisting of $\Q_p$-locally analytic vectors under the action of $G_p$. Then $\widehat{S}(U^p,L)_{\overline{\rho}}^{\mathrm{an}}$ is a locally analytic representation of $G_p$ and if we apply Emerton's Jacquet module functor with respect to $B_p$, $J_{B_p}(\widehat{S}(U^p,L)_{\overline{\rho}}^{\mathrm{an}})$ becomes an essentially admissible locally analytic representation of $T_p$ (\cite[Def. 6.4.9]{emerton2017locally}). The dual $J_{B_p}(\widehat{S}(U^p,L)^{\mathrm{an}}_{\overline{\rho}})'$ defines a coherent sheaf on the quasi-Stein space $\mathrm{Spf}(R_{\overline{\rho},S})^{\mathrm{rig}}\times \widehat{T}_{p,L}$. We define the eigenvariety $Y(U^p,\overline{\rho})$ to be the scheme-theoretical support of the coherent sheaf $J_{B_p}(\widehat{S}(U^p,L)_{\overline{\rho}}^{\mathrm{an}})'$ in $\mathrm{Spf}(R_{\overline{\rho},S})^{\mathrm{rig}}\times \widehat{T}_{p,L}$. An $L$-point $(\rho,\underline{\delta})\in \mathrm{Spf}(R_{\overline{\rho},S})^{\mathrm{rig}}\times \widehat{T}_{p,L}$ is in $Y(U^p,\overline{\rho})$ if and only if 
\[\Hom_{T_p}(\underline{\delta}, J_{B_p}(\widehat{S}(U^p,L)_{\overline{\rho}}[\fm_{\rho}]^{\mathrm{an}}))\neq 0\]
where $\fm_{\rho}$ is the maximal ideal of $R_{\overline{\rho},S}[\frac{1}{p}]$ corresponding to $\rho$ and $\widehat{S}(U^p,L)_{\overline{\rho}}[\fm_{\rho}]$ denotes the subspace of elements in $\widehat{S}(U^p,L)_{\overline{\rho}}$ annihilated by $\fm_{\rho}$.
\subsection{The companion points}\label{sec:companionpointsdesciption}
We will give the description of all companion points for a generic crystalline point. Suppose $(\rho,\underline{\delta})\in Y(U^p,\overline{\rho})(L)$. Let $\rho_v:=\rho\mid_{\cG_{F_{\widehat{v}}}}$ for $v\in S_p$. Set $\Sigma_v:=\{\tau: F_{\widetilde{v}}\hookrightarrow L\}$ for $v\in S_p$ and $\Sigma_p:=\cup_{v\in S_p} \Sigma_v$. Assume that for each $v\in S_p$, $\rho_v$ is crystalline. Then we have $\varphi$-modules $D_{\mathrm{cris}}(\rho_v)$ over $L\otimes_{\Q_p}F_{\widetilde{v},0}$, where $F_{\widetilde{v},0}$ is the maximal unramified subfield of $F_{\widetilde{v}}$. Take $\tau_{v,0}\in\Sigma_v$. Then $\varphi^{[F_{\widetilde{v},0}:\Q_p]}$ acts linearly on $D_{\mathrm{cris}}(\rho_v)\otimes_{L\otimes_{\Q_p}F_{\widetilde{v},0}, 1\otimes \tau_{v,0}}L$. Let $\{\varphi_{v,1},\cdots,\varphi_{v,n}\}$ be the multiset of eigenvalues of $\varphi^{[F_{\widetilde{v},0}:\Q_p]}$ which is independent of the choice of $\tau_{v,0}$. We say that $\rho$ is \emph{generic crystalline} if $ \varphi_{v,i}\varphi_{v,j}^{-1}\notin \{1, p^{[F_{\widetilde{v},0}:\Q_p]}\}$ for any $i\neq j$ and $v\in S_p$. A \emph{refinement} $\cR=(\cR_v)_{v\in S_p}$ for the generic crystalline representation $\rho$ is a choice of an ordering $\cR_v:\underline{\varphi}_v=(\varphi_{v,1},\cdots,\varphi_{v,n})$ of the $n$ different eigenvalues for all $v\in S_p$. Thus, $\rho$ has $(n!)^{|S_p|}$ different refinements.\par
Let $|\cdot|_{F_{\widetilde{v}}}$ be the norm of $F_{\widetilde{v}}$ such that $|p|_{F_{\widetilde{v}}}=p^{-[F_{\widetilde{v}}:\Q_p]}$. Denote by $\delta_{B_v}$ the smooth character $|\cdot|_{F_{\widetilde{v}}}^{n-1}\otimes\cdots\otimes|\cdot|_{F_{\widetilde{v}}}^{n-2i+1}\otimes\cdots \otimes |\cdot|^{1-n}_{F_{\widetilde{v}}}$ of $T_v\simeq (F_{\widetilde{v}}^{\times})^n$ and $\delta_{B_p}=\otimes_{v\in S_p}\delta_{B_v}$ the character of $T_p$. We define an automorphism $\iota=(\iota_{v})_{v\in S_p}$ of $\widehat{T}_{p,L}=\prod_{v\in S_p}\widehat{T}_{v,L}$ given by $\iota_v((\delta_{v,1},\cdots,\delta_{v,n}))=\delta_{B_v}(\delta_{v,1},\cdots,\delta_{v,i}\epsilon^{i-1}, \cdots,\delta_{v,n}\epsilon^{n-1})$ where $\epsilon$ denotes the cyclotomic characters.\par 
Let $\mathbf{h}=(\mathbf{h}_{\tau})_{\tau\in \Sigma_p}=(h_{\tau,1},\cdots,h_{\tau,n})_{\tau\in \Sigma_p}$ where $h_{\tau,1}\leq \cdots\leq h_{\tau,n}$ are the $\tau$-Hodge-Tate weights of $\rho_v$ if $\tau\in \Sigma_v$. Let $\mathcal{S}_n$ be the $n$-th symmetric group and act on the $n$-tuples $(h_{\tau,1},\cdots,h_{\tau,n})$ in standard ways. For $w=(w_v)_{v\in S_p}=(w_{\tau})_{v\in S_p,\tau\in \Sigma_v}\in (\mathcal{S}_n)^{\Sigma_p}$, define a character $\underline{\delta}_{\cR,w}:=\left(\iota_v\left(z^{w_v(\mathbf{h}_v)}\mathrm{unr}(\underline{\varphi}_v)\right)\right)_{v\in S_p}$
of $T_p$. Let $W_{P_p}=(W_{P_{\tau}})_{v\in S_p,\tau\in\Sigma_v}$ be the subgroup of $(\mathcal{S}_n)^{\Sigma_p}$ consisting of permutations that fix $\mathbf{h}$. Here $P_{\tau}$ denotes the parabolic subgroup of block upper-triangular matrices in $\GL_n$ with the Weyl group (of its Levi subgroup) identified with $W_{P_{\tau}}$. Set $D_{\mathrm{dR},\tau}(\rho_v):=D_{\mathrm{dR}}(\rho_v)\otimes_{L\otimes_{\Q_p} F_{\widetilde{v}},1\otimes\tau} L$. If we choose a basis $(e_1,\cdots, e_n)$ of $D_{\mathrm{dR},\tau}(\rho_v)$ of eigenvectors of $\varphi^{[F_{\widetilde{v},0},\Q_p]}$ with eigenvalues $(\varphi_{v,1},\cdots,\varphi_{v,n})$, then the Hodge-Tate filtration on $D_{\mathrm{dR},\tau}(\rho_v)$ corresponds to a point on the flag variety $\GL_n/P_{\tau}$ which lies in some Bruhat cell $B_{\tau}w_{\cR_{\tau}}P_{\tau}/P_{\tau}$ for some $w_{\cR_{\tau}}\in \mathcal{S}_n/W_{P_{\tau}}$ and $w_{\cR_{\tau}}$ is independent of scaling of the eigenvectors. Here $\cR$ signifies the refinement $\underline{\varphi}$. Let $w_{\cR}=(w_{\cR_{\tau}})_{v\in S_p,\tau\in \Sigma_v}\in (\mathcal{S}_n)^{\Sigma_p}/W_{P_p}$.\par   
Define a subset of points of $\widehat{T}_{p,L}$
\[W(\rho):=\{\underline{\delta}_{\cR,w} \mid w\in (\mathcal{S}_n)^{\Sigma_p}/W_{P_p}, w\geq w_{\cR}, \cR\text{ is a refinement of }\rho \} \] 
where $\geq$ denotes the usual Bruhat order on $\mathcal{S}_n$ (or its quotient).
Notice that there is a natural partition $W(\rho)=\coprod_{\cR}W_{\cR}(\rho)$ and $W(\rho)$ depends only on $\rho_v,v\in S_p$.\par
By the control of the companion points on the trianguline variety in the generic crystalline cases (\cite[\S4.2]{breuil2019local} and \cite[\S4.1]{wu2021local}), we have an inclusion $\{\underline{\delta}'\mid (\rho,\underline{\delta}')\in Y(U^p,\overline{\rho})\}\subset W(\rho)$. Below is our main theorem.
\begin{theorem}\label{theoremmaincrystalline}
	Let $(\rho,\underline{\delta})\in Y(U^p,\overline{\rho})(L)$ be a generic crystalline point as above and recall that we have assumed the Taylor-Wiles hypothesis (Assumption \ref{ass:taylorwiles}). Then 
	\[W(\rho)\subset \{\underline{\delta}'\mid (\rho,\underline{\delta}')\in Y(U^p,\overline{\rho})\}.\]
\end{theorem}
\begin{proof}
	We need the patched eigenvariety in \cite[\S3.2]{breuil2017interpretation}. For $v\in S_p$, let $R_{\overline{\rho}_v}'/\cO_L$ be the maximal reduced $\Z_p$-flat quotient of the framed deformation ring of $\overline{\rho}_{\widetilde{v}}$. We can similarly define $R_{\overline{\rho}_v}'$ for $v\in S\setminus S_p$. Let $K_p=\prod_{v\in S_p}\GL_n(\cO_{F_{\widetilde{v}}})\subset \prod_{v\in S_p}\GL_n(F_{\widetilde{v}})\simeq G_p$. Recall that under the Taylor-Wiles assumption, there are some positive integers $g$ and $q$, a patching module $M_{\infty}$ in \cite{caraiani2016patching} over the ring $R_{\infty}=\widehat{\otimes}_{v\in S}R_{\overline{\rho}_v}'[[x_1,\cdots,x_g]]$, an $\cO_L$-morphism $S_{\infty}:=\cO_L[[y_1,\cdots,y_q]]\rightarrow R_{\infty}$ and a surjection $R_{\infty}/\mathfrak{a}\rightarrow R_{\overline{\rho},\mathcal{S}}$ of completed local rings over $\cO_L$ where $\mathfrak{a}=(y_1,\cdots,y_q)$, such that $M_{\infty}$ is a finite projective $S_{\infty}[[K_p]]$-module and $\Pi_{\infty}:=\Hom_{\cO_L}^{\mathrm{cont}}(M_{\infty},L)$ is a $R_{\infty}$-admissible Banach representation of $G_p$ with an isomorphism $\Pi_{\infty}[\mathfrak{a}]\simeq \widehat{S}(U^p,L)_{\overline{\rho}}$ that is compatible with the actions of $R_{\infty}/\mathfrak{a}$ and $R_{\overline{\rho},\mathcal{S}}$ (the action of $R_{\overline{\rho},S}$ factors through the quotient $ R_{\overline{\rho},\mathcal{S}}$). Write $\Pi_{\infty}^{\mathrm{an}}$ for the subspace of locally $R_{\infty}$-analytic vectors in $\Pi_{\infty}$ (\cite[D\'ef. 3.2]{breuil2017interpretation}). The patched eigenvariety $X_p(\overline{\rho})$ is the support of $J_{B_p}(\Pi_{\infty}^{\mathrm{an}})'$ inside $\Spf(R_{\infty})^{\mathrm{rig}}\times\widehat{T}_{p,L}\simeq \Spf(\widehat{\otimes}_{v\in S_p}R_{\overline{\rho}_v}')^{\mathrm{rig}}\times \Spf(\widehat{\otimes}_{v\in S\setminus S_p}R_{\overline{\rho}_v}')^{\mathrm{rig}}\times \Spf(\cO_L[[x_1,\cdots,x_g]])^{\mathrm{rig}}\times\widehat{T}_{p,L}=:\fX_{\overline{\rho}_p}\times \fX_{\overline{\rho}^p}\times \mathbb{U}^g\times\widehat{T}_{p,L}$. By \cite[Thm. 3.21, \S4.1]{breuil2017interpretation}, we have closed embeddings
	\[Y(U^p,\overline{\rho})\hookrightarrow X_p(\overline{\rho})\hookrightarrow \iota\left(X_{\mathrm{tri}}(\overline{\rho}_p)\right)\times (\mathfrak{X}_{\overline{\rho}^p}\times\mathbb{U}^g) \subset \fX_{\overline{\rho}_p}\times\widehat{T}_{p,L}\times \fX_{\overline{\rho}^p}\times \mathbb{U}^g\]
	where $X_{\mathrm{tri}}(\overline{\rho}_p)=\prod_{v\in S_p}X_{\mathrm{tri}}(\overline{\rho}_v)$ and $\iota$ is extended to an automorphism of $\fX_{\overline{\rho}_p}\times \widehat{T}_{p,L}$ by base change. Moreover, $X_p(\overline{\rho})$ is equidimensional and is identified with a union of irreducible components of $\iota\left(X_{\mathrm{tri}}(\overline{\rho}_p)\right)\times (\mathfrak{X}_{\overline{\rho}^p}\times\mathbb{U}^g)$ under the above closed embedding. By the argument as in the first steps of the proof of \cite[Thm. 5.3.3]{breuil2019local}, we are reduced to prove the lemma below.
\begin{lemma}
	Assume that a point $\left((\rho_p=(\rho_v)_{v\in S_p},\iota(\underline{\delta})),y\right)\in\iota\left(X_{\mathrm{tri}}(\overline{\rho}_p)\right)\times (\mathfrak{X}_{\overline{\rho}^p}\times\mathbb{U}^g)$ is in $X_p(\overline{\rho})(L)$ where each $\rho_v, v\in S_p$ is generic crystalline. Then $\left((\rho_p,\underline{\delta}_{\cR,w}),z\right)\in X_p(\overline{\rho})(L)$ if and only if $w\geq w_{\cR}$ in $(\mathcal{S}_n)^{\Sigma_p}/W_{P_p}$ where $\cR$ denotes refinements of $\rho_p$.	
\end{lemma}
	Now we prove the lemma. By \cite[Thm. 4.10]{wu2021local}, we may assume $\left((\rho_p=(\rho_v)_{v\in S_p},\iota(\underline{\delta})),z\right)\in\iota\left(U_{\mathrm{tri}}(\overline{\rho}_p)\right)\times (\mathfrak{X}_{\overline{\rho}^p}\times\mathbb{U}^g)$. Suppose that $\underline{\delta}$ corresponds to a refinement $\underline{\varphi}=(\underline{\varphi}_v)_{v\in S_p}$. We need to prove that the companion points for other refinements exist on the eigenvariety. We only need to prove the existence of companion points for an arbitrary refinement $\underline{\varphi}'$ such that $\underline{\varphi}_v'=\underline{\varphi}_v$ for all $v\neq v_0$ for some $v_0$ and $\varphi_{v_0}'$ is the refinement permuting $\varphi_{v_0,i_0}$ and $\varphi_{v_0,i_0+1}$. By Proposition \ref{prop:key}, there exists a subset $Z\subset X_{\mathrm{tri}}(\overline{\rho}_{v_0})$ that quasi-accumulates at $(\rho_{v_0},\underline{\delta}_{v_0})$ consisting of generic regular crystalline points and their local companion points $\zeta(Z)$ quasi-accumulates at $\zeta((\rho_{v_0},\underline{\delta}_{v_0}))$ which is a local companion point of $(\rho_{v_0},\underline{\delta}_{v_0})$ for the refinement $\underline{\varphi}_{v_0}'$. Since $U_{\mathrm{tri}}(\overline{\rho}_p)$ is smooth at $(\rho_p,\underline{\delta})$, we may assume every $(z, (\rho_{v},\underline{\delta}_v)_{v\neq v_0}),z\in Z$ is contained in the same irreducible component of $X_{\mathrm{tri}}(\rho_p)$ with $((\rho_v,\underline{\delta}_v)_{v\in S_p})$. In particular for any $z\in Z$, $\left(\iota(z,(\rho_v,\underline{\delta}_v)_{v\neq v_0}),y\right)\in X_p(\overline{\rho})$.  By \cite[Thm. 5.5]{breuil2017smoothness}, the classicality (which follows from \cite[Prop. 4.9]{wu2021local}, but essentially \cite[Thm. 3.9]{breuil2017smoothness} is enough for us, and the classicality is only partial for $v_0$), and the discussions in the beginning of \cite[\S5.3]{breuil2019local}, the companion points $\left(\iota(\zeta(z),(\rho_v,\underline{\delta}_v)_{v\neq v_0}),y\right)$ are in $X_p(\overline{\rho})$ and quasi-accumulate at the point $\left(\iota(\zeta((\rho_{v_0},\underline{\delta}_{v_0})),(\rho_v,\underline{\delta}_v)_{v\neq v_0}),y\right)$ in $\fX_{\overline{\rho}_p}\times \widehat{T}_{p,L}\times (\mathfrak{X}_{\overline{\rho}^p}\times\mathbb{U}^g)$. Hence $\left(\iota(\zeta((\rho_{v_0},\underline{\delta}_{v_0})),(\rho_v,\underline{\delta}_v)_{v\neq v_0}),y\right)\in X_p(\overline{\rho})$ by Lemma \ref{lem: quasi-accumulates} and Lemma \ref{lem:quasi-accumulateszariskiclosure}.
\end{proof}
\subsection{Locally analytic socle conjecture}
Let $(\rho,\underline{\delta})\in Y(U^p,\overline{\rho})(L)$ be generic crystalline as before. We write $\lambda=(\lambda_{\tau})_{\tau\in\Sigma_v,v\in S_p}\in (\Z^n)^{\Sigma_p}$ where $\lambda_{\tau}=(\lambda_{\tau,1},\cdots,\lambda_{\tau,n}):= (h_{\tau,n},\cdots, h_{\tau,i}+n-i,\cdots,h_{\tau,1}+n-1)$. We identify the base change to $L$ of the $\Q_p$-Lie algebra of $G_p$ with $\fg:=\prod_{\tau\in \Sigma_p}\mathfrak{gl}_{n/L}$. Let $\overline{\fb}=\prod_{\tau\in \Sigma_p}\overline{\fb}_{\tau}$ be the Borel subalgebra of $\fg$ of lower triangular matrices and $\ft=\prod_{\tau\in \Sigma_p} \ft_{\tau}$ be the Cartan subalgebra of diagonal matrices. We view $\lambda$ as a weight of $\ft$ and extend it to $\overline{\fb}$. For a weight $\mu$ of $\ft$, let $\overline{L}(\mu)$ be the irreducible $\fg$-module with the highest weight $\mu$ in the BGG category attached to $\overline{\fb}$. For a refinement $\cR$ of $\rho$, we write $\underline{\delta}_{\cR,\mathrm{sm}}$ for the smooth part of $\delta_{\cR,w}$, that is $\underline{\delta}_{\cR,\mathrm{sm}}\underline{\delta}_{\cR,w}^{-1}$ is an algebraic character of $T_p$. Notice that $\underline{\delta}_{\cR,\mathrm{sm}}$ is independent of $w$. Let $\overline{B}_p$ be the opposite Borel subgroup of $B_p$ in $G_p$. Recall by Orlik-Strauch's theory \cite{orlik2015jordan}, we have topologically irreducible admissible locally analytic representations ${\cF_{\overline{B}_p}^{G_p}(\overline{L}(-ww_0\cdot \lambda), \underline{\delta}_{\cR,\mathrm{sm}}\delta_{B_p}^{-1})}$, see e.g. \cite[\S4.3]{wu2021local}. Here $w_0$ is the longest element in $\mathcal{S}_n^{\Sigma_p}$ and $ww_0\cdot\lambda$ denotes the usual dot action. By \cite[Prop. 4.9]{wu2021local}, we have the following corollary of Theorem \ref{theoremmaincrystalline} on the locally analytic socle conjecture.
\begin{corollary}\label{cor:socle}
	Under the assumptions and notation of Theorem \ref{theoremmaincrystalline}, there is an injection 
	\begin{equation}\label{equa:socle}
		\cF_{\overline{B}_p}^{G_p}(\overline{L}(-ww_0\cdot \lambda), \underline{\delta}_{\cR,\mathrm{sm}}\delta_{B_p}^{-1})\hookrightarrow \widehat{S}(U^p,L)_{\overline{\rho}}[\mathfrak{m}_{\rho}]^{\mathrm{an}}
	\end{equation}
	of locally analytic representations of $G_p$ for all refinements $\cR$ of $\rho$ and $w\in \mathcal{S}_n^{\Sigma_p}/W_{P_p}, w\geq w_{\cR}$.
\end{corollary}
\begin{remark}
	Assuming that, in the situation of Theorem \ref{theoremmaincrystalline} and Corollary \ref{cor:socle}, the Hodge-Tate weights of $\rho_v$ satisfy that $h_{\tau,i}\neq h_{\tau,j}$ for all $i\neq j$ and $\tau\in \Sigma_p$, then there exists a finite length admissible locally analytic representation $\Pi(\rho_p)^{\mathrm{fs}}:=\widehat{\otimes}_{v\in S_p}\Pi(\rho_v)^{\mathrm{fs}}$ of $G_p$ in \cite{breuil2020towards} such that the $G_p$-socle of $\Pi(\rho_p)^{\mathrm{fs}}$ coincides with the finite direct sum of pairwise non-isomorphic irreducible admissible locally analytic representations of $G_p$ that are isomorphic to one of those in the left-hand side of (\ref{equa:socle}) and there exists an injection $\Pi(\rho_p)^{\mathrm{fs}}\hookrightarrow \widehat{S}(U^p,L)_{\overline{\rho}}[\mathfrak{m}_{\rho}]^{\mathrm{an}}$ (\cite[Thm. 1.1]{breuil2020towards}). The representation $\Pi(\rho_p)^{\mathrm{fs}}$ is called the ``finite slope part'' since it is constructed from principal series (thus has Jordan-Hölder factors of the type of Orlik-Strauch). Using Corollary \ref{cor:socle}, similar result still holds without the regular assumption on Hodge-Tate weights. One just need to notice that \cite[Prop. 4.8]{breuil2020towards} is proved without any assumption on the regularity of weights and we can define $\Pi(\rho_v)^{\mathrm{fs}}$ in non-regular cases in the same way as \cite[Def. 5.7]{breuil2020towards}. Then the proof of \cite[Thm. 5.12]{breuil2020towards} applies with minor modifications.
\end{remark}
\bibliography{bibfile.bib} 
\bibliographystyle{alpha}
\end{document}